\documentclass[12pt]{article}
\usepackage{algorithm}
\usepackage{algpseudocode}
\usepackage{amsfonts}
\usepackage{amsmath}
\usepackage{amssymb} 
\usepackage{amsthm} 
\usepackage[title]{appendix} 
\usepackage{autobreak}
\usepackage{bigdelim} 
\usepackage{bm}
\usepackage{breqn}
\usepackage{cancel} 
\usepackage{cases}
\usepackage{chngcntr}
\usepackage{colortbl} 
\usepackage{comment}
\usepackage{empheq} 
\usepackage{enumerate}
\usepackage{fancyhdr} 
\usepackage{fancyvrb} 
\usepackage{framed} 
\usepackage{graphicx}
\usepackage[dvipdfmx]{hyperref} 
\hypersetup{
  colorlinks,
  linkcolor={red!50!black},
  citecolor={blue!50!black},
  urlcolor={blue!80!black}
}
\usepackage{mathrsfs} 
\usepackage{multirow} 
\usepackage{setspace} 
\usepackage{subfig} 
\usepackage[dvipdfmx]{pdflscape} 
\usepackage{pgfplots} 
\usepackage{pxrubrica} 
\usepackage{url}

\usepackage[top=35truemm,bottom=35truemm,left=30truemm,right=30truemm]{geometry} 
\makeatletter

\@addtoreset{equation}{section}
\makeatother

\algtext*{EndWhile}
\algtext*{EndIf}
\algtext*{EndFor}

\newtheorem{theorem}{Theorem}[section]
\newtheorem{proposition}[theorem]{Proposition}%
\newtheorem{corollary}[theorem]{Corollary}
\newtheorem{lemma}[theorem]{Lemma}
\newtheorem{conjecture}[theorem]{Conjecture}
\newtheorem{example}[theorem]{Example}%
\newtheorem{remark}[theorem]{Remark}%
\newtheorem{definition}[theorem]{Definition}%

\DeclareMathOperator*{\rank}{rank}
\DeclareMathOperator*{\conv}{conv}
\DeclareMathOperator*{\cl}{cl}

\DeclareMathOperator*{\setint}{int}

\DeclareMathOperator*{\Ext}{Ext}

\DeclareMathOperator*{\Diag}{Diag}
\DeclareMathOperator*{\svec}{svec}

\DeclareMathOperator*{\smat}{smat}
\DeclareMathOperator{\subjectto}{subject~to}

\newcommand{\minimize}{\mathop{\rm minimize}\limits}
\newcommand{\maximize}{\mathop{\rm maximize}\limits}
\newcommand{\relmiddle}[1]{\mathrel{}\middle#1\mathrel{}}
\newcommand{\overbar}[1]{\mkern 1.5mu\overline{\mkern-1.5mu#1\mkern-1.5mu}\mkern 1.5mu}

\def\DNN{\mathcal{DNN}}
\def\CP{\mathcal{CP}}
\def\COP{\mathcal{COP}}

\newcommand{\RNum}[1]{\uppercase\expandafter{\romannumeral #1\relax}} 
\newcommand{\Rnum}[1]{\lowercase\expandafter{\romannumeral #1\relax}} 

\title{Generalizations of doubly nonnegative cones\\and their comparison}

\makeatletter
\let\@fnsymbol\@arabic
\makeatother

\author{
\normalsize
    Mitsuhiro Nishijima\thanks{Department of Industrial Engineering and Economics,
        Tokyo Institute of Technology, 2-12-1 Ookayama, Meguro-ku, 1528552, Tokyo, Japan.
        ({\tt nishijima.m.ae@m.titech.ac.jp, nakata.k.ac@m.titech.ac.jp}).}
\and
\normalsize
        Kazuhide Nakata\footnotemark[1]
        }

\begin{document}
\maketitle

\begin{abstract}\noindent
In this study, we examine the various extensions of the doubly nonnegative (DNN) cone, frequently used in completely positive programming (CPP) to achieve a tighter relaxation than the positive semidefinite cone.
To provide tighter relaxation for generalized CPP (GCPP) than the positive semidefinite cone, inner-approximation hierarchies of the generalized copositive cone are exploited to obtain two generalized DNN (GDNN) cones from the DNN cone.
This study conducts theoretical and numerical comparisons to assess the relaxation strengths of the two GDNN cones over the direct products of a nonnegative orthant and second-order or positive semidefinite cones.
These comparisons also include an analysis of the existing GDNN cone proposed by Burer and Dong.
The findings from solving several GDNN programming relaxation problems for a GCPP problem demonstrate that the three GDNN cones provide significantly tighter bounds for GCPP than the positive semidefinite cone.
\end{abstract}
\vspace{0.5cm}


\section{Introduction}\label{sec:intro}
Completely positive programming (CPP), also known as copositive (COP) programming, is a class of conic programming with a completely positive (CP) cone or its dual, a COP cone.
CPP has received significant attention over the last few decades because it can be used to reformulate many NP-hard problems, such as the standard quadratic program (QP)~\cite{BDde2000}, the quadratic assignment problem~\cite{PR2009}, and nonconvex QP with binary and continuous variables~\cite{Burer2009}, as convex programming in a unified manner.

Recently, the concepts of CP and COP cones have been generalized in many ways. Examples include the generalized CP (GCP) $\CP(\mathbb{K})$ and COP (GCOP) $\COP(\mathbb{K})$ cones over a closed cone $\mathbb{K} \subseteq \mathbb{R}^n$,\footnote{
Some prior literature~\cite{BMP2016,Burer2012,GS2013} including the authors' paper~\cite{NN20XX}, do not use the term ``generalized'' when referring to the sets $\CP(\mathbb{K})$ and $\COP(\mathbb{K})$.
Nevertheless, we use this term to specifically highlight the distinction: whether $\mathbb{K}$ represents a nonnegative orthant or otherwise.} which are defined as
\begin{align}
\CP(\mathbb{K}) &= \left\{\sum_{i=1}^m\bm{a}_i\bm{a}_i^\top \relmiddle| \bm{a}_i\in \mathbb{K} \text{ for all $i = 1,\dots,m$}\right\}, \label{eq:def_CP}\\
\COP(\mathbb{K}) &\coloneqq \{\bm{A} \mid \text{$\bm{A}$ is a symmetric matrix and }\bm{x}^\top \bm{A}\bm{x}\ge 0\ \text{for all $\bm{x}\in \mathbb{K}$}\},\label{eq:def_COP}
\end{align}
respectively.
GCP and GCOP cones over a nonnegative orthant respectively reduce to CP and COP cones.
Hereafter, ``over $\mathbb{K}$" can be omitted when we need not specify it.
GCP programming (GCPP) is a class of conic programming with GCP or GCOP cones.
Because GCPP includes CPP, it can express not only the aforementioned problems, but also many other NP-hard problems.
For example, a rank-constrained semidefinite programming (SDP) problem~\cite{BMP2016}, which includes sensor network localization~\cite{NN2022} and the max-cut problem~\cite{GW1995} as a class, can be formulated as a problem with a GCP cone over a direct product of a nonnegative orthant and three positive semidefinite cones. To the best of our knowledge, rank-constrained SDP is not known to be representable as CPP. Therefore, solving GCPP but not CPP represents an important task.
In addition, under an assumption, the quadratically constrained QP (QCQP) can be equivalently reformulated as a problem with a GCP cone over a direct product of a nonnegative orthant and second-order or positive semidefinite cones~\cite{BD2012}.
Other NP-hard problems that GCPP can represent include nonconvex conic QP with binary and continuous variables~\cite{Burer2012} (e.g., a variable selection problem in linear regression~\cite{MT2015}) and $k$-means clustering~\cite{PH2018}.
However, as GCPP can represent such formidable problems equivalently, it is also difficult to solve directly.

Some approximation hierarchies for the GCOP cone have been developed to solve GCPP approximately~\cite{NN20XX,ZVP2006}.
In a prior work, we solved approximation problems obtained by replacing the GCOP cone with its inner-approximation hierarchies~\cite{NN20XX}.
Our results showed that the optimal values of approximation problems with the hierarchy provided by Zuluaga, Vera, and Pe\~{n}a~\cite{ZVP2006}, referred to as the Zuluaga--Vera--Pe\~{n}a (ZVP)-type hierarchy, and those of problems with a hierarchy we provided, referred to as a Nishijima--Nakata (NN)-type hierarchy, were the same and almost identical to those of the original problems~\cite{NN20XX}.
This result poses a problem, i.e., whether the ZVP- and NN-type hierarchies are the same.

Of note, the two hierarchies are generalizations of the inner-approximation hierarchy provided by Parrilo~\cite{Parrilo2000} for the COP cone, and the dual cone of the zeroth level of Parrilo's hierarchy is known to be the \emph{doubly nonnegative (DNN) cone}, that is, the set of positive semidefinite matrices with only nonnegative elements~\cite[Section~5.3]{Parrilo2000}.
The DNN cone is frequently used because DNN programming relaxation can provide tighter bounds for CPP than SDP relaxation~\cite{PR2009,YM2010}.
Thus, generalizing the DNN cone to provide tighter bounds for GCPP than SDP relaxation is a natural approach.
Hence, we refer to the dual cone of the zeroth level of the ZVP- and NN-type hierarchies as the ZVP- and NN-type generalized DNN (GDNN) cones, respectively.
Regarding these sets as GDNN cones is justified in Section~\ref{sec:DNN}.

In the present work, we provide theoretical and numerical comparisons of the strength of relaxation by the above two GDNN cones as well as that proposed by Burer and Dong~\cite{BD2012} (the BD-type GDNN cone).
To compare these cones, we focus on the two specific cases in which (\Rnum{1}) $\mathbb{K}$ is a direct product of a nonnegative orthant and second-order cones, and (\Rnum{2}) $\mathbb{K}$ is a direct product of a nonnegative orthant and positive semidefinite cones.
As can be seen from the examples of rank-constrained SDP and QCQP, these two cases naturally arise in the context of optimization.
Specifically, we investigate the inclusion relationship between the three GDNN cones because inclusion is a crucial factor in the strength of relaxation.
When $\mathbb{K}$ is a direct product of a nonnegative orthant and second-order cones, no theoretical inclusion relationship is obtained between ZVP- and BD-type GDNN cones (see Examples~\ref{ex:inZVP_notinBD} and \ref{ex:inBD_notinZVP}), and numerically, the ZVP-type GDNN cone provides better relaxation than BD-type cone (Section~\ref{subsec:misocp}).
By contrast, when $\mathbb{K}$ is a direct product of a nonnegative orthant and positive semidefinite cones, the ZVP-type GDNN cone provides worse relaxation than the BD-type cone numerically (Section~\ref{subsec:max-cut}). In both cases, theoretically, the NN-type GDNN cone is strictly included in the ZVP-type cone (Theorems~\ref{thm:DNN_inclusion_NN_ZVP} and \ref{thm:DNN_inclusion_NN_ZVP_SDC}).
Solutions to the GDNN programming (GDNNP) relaxation of GCPP problems show that the three GDNN cones provide much tighter relaxation for GCPP than the positive semidefinite cone (Section~\ref{sec:experiment}).

In Section~\ref{sec:preliminaries}, we introduce the notation and concepts used throughout this paper.
In Section~\ref{sec:DNN}, we introduce two new GDNN cones and describe the BD-type cone.
In Section~\ref{sec:relationship}, we discuss the theoretical properties of the three GDNN cones.
In Section~\ref{sec:experiment}, we describe the results of numerical experiments conducted to investigate the effectiveness of relaxation by these cones.
In Section~\ref{sec:conclusion}, we present our conclusions and suggest potential future directions of research.

\section{Preliminaries}
\label{sec:preliminaries}
\subsection{Notation}
\label{subsec:notation}
We use $\mathbb{N}$, $\mathbb{R}$, and $\mathbb{S}^n$ to denote the set of nonnegative integers, the set of real numbers, and the space of real $n\times n$ symmetric matrices, respectively.
For $n\in\mathbb{N}$, let $T_n \coloneqq n(n+1)/2$.
We use $\bm{e}_i$, $\bm{0}$, and $\bm{1}$ to represent the vector with the $i$th element being $1$ and all other elements $0$, the zero vector, and the vector with all elements $1$, respectively.
In addition, we use $\bm{O}$ and $\bm{I}$ to represent the zero matrix and the identity matrix, respectively.
We sometimes use a subscript such as $\bm{0}_n$ and $\bm{I}_n$ to specify the size.
All vectors that appear in this paper are column vectors.
However, for notational convenience, the difference between a column and a row is omitted if it is clear from the context.
Let $\bm{A}\in\mathbb{S}^n$.
The $(i,j)$th element of $\bm{A}$ is expressed as $A_{i,j}$.
Moreover, for index sets $\mathcal{I},\mathcal{J}\subseteq \{1,\dots,n\}$, $\bm{A}_{\mathcal{I},\mathcal{J}}$ denotes the submatrix obtained by extracting the rows of $\bm{A}$ indexed by $\mathcal{I}$, and the columns indexed by $\mathcal{J}$.
We also refer to $\bm{A}_{\mathcal{I},\mathcal{J}}$ as the $(\mathcal{I},\mathcal{J})$th element of $\bm{A}$.
The space $\mathbb{R}^n$ is endowed with the usual transpose inner product, and $\|\cdot\|$ denotes the induced norm (2-norm).
The space $\mathbb{S}^n$ is endowed with the trace inner product defined by $\langle \bm{X},\bm{Y}\rangle \coloneqq \sum_{i,j=1}^n X_{i,j}Y_{i,j}$ for $\bm{X},\bm{Y}\in\mathbb{S}^n$. We use $S^n$ to denote the $n$-dimensional unit sphere in $\mathbb{R}^{n+1}$; i.e., $S^n = \{\bm{x}\in\mathbb{R}^{n+1}\mid \|\bm{x}\| = 1\}$.
For matrices $\bm{X}_i\in\mathbb{S}^{n_i}\ (i = 1,\dots,k)$, $\Diag(\bm{X}_1,\dots,\bm{X}_k)$ denotes the block diagonal matrix with the $i$th block $\bm{X}_i$.
For a set $\mathcal{X}$, we use $|\mathcal{X}|$, $\conv(\mathcal{X})$, $\cl(\mathcal{X})$, and $\setint(\mathcal{X})$ to denote the cardinality, convex hull, closure, and interior of $\mathcal{X}$, respectively.

The set $\mathcal{K}$ in a finite-dimensional real vector space is called a cone if $\alpha x\in\mathcal{K}$ for all $\alpha>0$ and $x\in\mathcal{K}$.
For a cone $\mathcal{K}$ in a finite-dimensional real inner product space, $\mathcal{K}^*$ denotes its dual cone; i.e., the set of $x$ such that the inner product between $x$ and $y$ is greater than or equal to $0$ for all $y\in\mathcal{K}$.
A cone $\mathcal{K}$ is said to be pointed if it contains no lines.
For a closed convex cone $\mathcal{K}$, $\Ext(\mathcal{K})$ denotes the set of elements generating extreme rays (one-dimensional faces) of $\mathcal{K}$.
The following properties of a cone and its dual are established:
\begin{theorem}[{\cite[Section~2.6.1]{BV2004}}]
Let $\mathcal{K}$, $\mathcal{K}_1$, and $\mathcal{K}_2$ be cones in a finite-dimensional real inner product space.
Then,
\begin{enumerate}[(i)]
\item $\mathcal{K}^*$ is a closed convex cone.
\item If $\mathcal{K}$ is a pointed closed convex cone, $\mathcal{K}^*$ has a nonempty interior.
\item If $\mathcal{K}$ is a closed convex cone, $(\mathcal{K}^*)^* = \mathcal{K}$.
\item If $\mathcal{K}_1 \subseteq \mathcal{K}_2$, $\mathcal{K}_2^* \subseteq \mathcal{K}_1^*$.
\end{enumerate}
\end{theorem}

We use $\mathbb{R}_+^n$ and $\mathbb{S}_+^n$ to denote the set of $n$-dimensional nonnegative vectors (nonnegative orthant) and the set of $n\times n$ symmetric positive semidefinite matrices (positive semidefinite cone), respectively.
Moreover, for $\bm{S}\in\mathbb{S}^n$, let $\mathbb{R}_+\bm{S} \coloneqq \{\alpha\bm{S}\mid \alpha\ge 0\}$.

We use $H^{n,m}$ to denote the set of homogeneous polynomials in $n$ variables of degree $m$ with real coefficients.
Let $\mathbb{I}_{=m}^n \coloneqq \{\bm{\alpha} \in \mathbb{N}^n\mid \sum_{i=1}^n\alpha_i=m\}$. $\mathbb{R}^{\mathbb{I}_{=m}^n}$ and $\mathbb{S}^{\mathbb{I}_{=m}^n}$ denote the $|\mathbb{I}_{=m}^n|$-dimensional Euclidean space with elements indexed by $\mathbb{I}_{=m}^n$, and the space of real $|\mathbb{I}_{=m}^n| \times |\mathbb{I}_{=m}^n|$ symmetric matrices with columns and rows indexed by $\mathbb{I}_{=m}^n$, respectively. $\mathbb{S}_+^{\mathbb{I}_{=m}^n}$ denotes the set of positive semidefinite matrices in $\mathbb{S}^{\mathbb{I}_{=m}^n}$.
$\Sigma^{n,2m}$ denotes the set of sums of squares (SOS) of elements in $H^{n,m}$; i.e., $\Sigma^{n,2m} = \conv\{\theta^2\mid\theta\in H^{n,m}\}$.
For $\bm{y}\in\mathbb{R}^{\mathbb{I}_{=2m}^n}$, $\bm{M}^{n,m}(\bm{y})\in\mathbb{S}^{\mathbb{I}_{=m}^n}$ is the matrix with $(\bm{\alpha},\bm{\alpha}')$th element $y_{\bm{\alpha}+\bm{\alpha}'}$ for $\bm{\alpha},\bm{\alpha}'\in \mathbb{I}_{=m}^n$.
For example, in the case of $n = 2$ and $m = 3$, we have
\begin{equation*}
\bm{M}^{2,3}(\bm{y}) = \bordermatrix{
& (3,0) &(2,1) & (1,2) & (0,3)\cr
(3,0) & y_{(6,6)} & y_{(5,1)} & y_{(4,2)} & y_{(3,3)} \cr
(2,1) & y_{(5,1)} & y_{(4,2)} & y_{(3,3)} & y_{(2,4)} \cr
(1,2) & y_{(4,2)} & y_{(3,3)} & y_{(2,4)} & y_{(1,5)} \cr
(0,3) & y_{(3,3)} & y_{(2,4)} & y_{(1,5)} & y_{(0,6)} \cr
}.
\end{equation*}
Then, we define $\mathcal{M}^{n,2m} \coloneqq \{\bm{y}\in \mathbb{R}^{\mathbb{I}_{=2m}^n}\mid \bm{M}^{n,m}(\bm{y})\in \mathbb{S}_+^{\mathbb{I}_{=m}^n}\}$, which is a closed convex cone.

For a closed cone $\mathbb{K}$ in $\mathbb{R}^n$, we define the GCP cone $\CP(\mathbb{K})$ as \eqref{eq:def_CP} and GCOP cone $\COP(\mathbb{K})$ as \eqref{eq:def_COP}, respectively.
By definition, $\CP(\mathbb{K}) \subseteq \mathbb{S}_+^n$ holds.
Moreover, $\CP(\mathbb{K})$ and $\COP(\mathbb{K})$ are mutually dual~\cite{SZ2003}.
We define $\mathcal{N}^n$ as the set of $n\times n$ symmetric matrices with only nonnegative elements, with the DNN cone $\mathbb{S}_+^n\cap \mathcal{N}^n$ denoted by $\DNN^n$.
By definition, $\CP(\mathbb{R}_+^n) \subseteq \DNN^n \subseteq \mathbb{S}_+^n$.
In addition, $\DNN^n$ and $\mathbb{S}_+^n + \mathcal{N}^n$ are mutually dual~\cite[Theorem~1.167]{SB2021}.
In particular, $\mathbb{S}_+^n + \mathcal{N}^n$ is a closed convex cone.

\subsection{Euclidean Jordan algebra and symmetric cone}
\label{subsec:EJA}
A finite-dimensional real vector space $\mathbb{E}$ equipped with a bilinear mapping (product) denoted by $\circ$ is called a Jordan algebra if it satisfies the following properties for all $x,y\in\mathbb{E}$:
\begin{enumerate}[(J1)]
\item $x \circ y = y \circ x$,
\item $x\circ ((x\circ x)\circ y) = (x\circ x)\circ(x\circ y)$.
\end{enumerate}
We assume in this paper that $\mathbb{E}$ has a (unique) identity element $e$ for the product. A Jordan algebra is called Euclidean if there exists an associative inner product $\bullet$; i.e., $(x\circ y)\bullet z = x\bullet(y\circ z)$ for all $x,y,z\in\mathbb{E}$.

We define the symmetric cone $\mathbb{E}_+$ associated with the Euclidean Jordan algebra $(\mathbb{E},\circ,\bullet)$ as $\{x\circ x\mid x\in\mathbb{E}\}$. Symmetric cones are known to be self-dual---i.e., $(\mathbb{E}_+)^* = \mathbb{E}_+$---and are therefore pointed closed convex cones with a nonempty interior.
Specifically, the identity element $e$ of $\mathbb{E}$ is in $\setint(\mathbb{E}_+)$~\cite[Theorem~\mbox{\RNum{3}.2.1}]{FK1994}.

The following three examples represent typical Euclidean Jordan algebras that appear frequently throughout this paper.
\begin{example}
[nonnegative orthant]\label{ex:nno}
Let $\mathbb{E}$ be the space $\mathbb{R}^n$.
If we define $\bm{x}\circ \bm{y} \coloneqq (x_1y_1,\dots,x_ny_n)$ and $\bm{x}\bullet \bm{y} \coloneqq \bm{x}^\top\bm{y}$ for $\bm{x},\bm{y}\in \mathbb{E}$, then $(\mathbb{E},\circ,\bullet)$ is a Euclidean Jordan algebra, and the associated symmetric cone $\mathbb{E}_+$ is $\mathbb{R}_+^n$. The identity element of the Euclidean Jordan algebra is $\bm{1}_n$.
\end{example}

\begin{example}
[second-order cone]\label{ex:soc}
Let $\mathbb{E}$ be the space $\mathbb{R}^n$ and for each $\bm{x} \in \mathbb{E}$, we express $(x_2,\dots,x_n)$ as $\bm{x}_{2:n}$, so that $\bm{x}$ can be simply expressed as $(x_1,\bm{x}_{2:n})$.
If we define $\bm{x}\circ \bm{y} \coloneqq (\bm{x}^\top\bm{y},x_1\bm{y}_{2:n} + y_1\bm{x}_{2:n})$ and $\bm{x}\bullet\bm{y} \coloneqq \bm{x}^\top\bm{y}$ for $\bm{x},\bm{y}\in \mathbb{E}$, then $(\mathbb{E},\circ,\bullet)$ is a Euclidean Jordan algebra, and the associated symmetric cone $\mathbb{E}_+$ is the (standard) second-order cone $\mathbb{L}^n$ defined as
\begin{equation*}
\mathbb{L}^n \coloneqq \begin{cases}
\mathbb{R}_+ & (n = 1),\\
\{\bm{x}\in\mathbb{R}^n \mid x_1 \ge \|\bm{x}_{2:n}\|\} & (n\ge 2).
\end{cases}
\end{equation*}
The identity element of the Euclidean Jordan algebra is $(1,\bm{0}_{n-1})$.
\end{example}

\begin{example}[positive semidefinite cone]\label{ex:sdc} Consider the space $\mathbb{S}^n$ of real $n\times n$ symmetric matrices.
If we define $\bm{X}\mathbin{\lozenge} \bm{Y} \coloneqq (\bm{X}\bm{Y} + \bm{Y}\bm{X})/2$ and $\bm{X}\mathbin{\blacklozenge} \bm{Y} \coloneqq \langle \bm{X},\bm{Y}\rangle$ for $\bm{X},\bm{Y}\in \mathbb{S}^n$, then $(\mathbb{S}^n,\lozenge,\blacklozenge)$ is a Euclidean Jordan algebra, and the associated symmetric cone is the positive semidefinite cone $\mathbb{S}_+^n$.
The identity element of the Euclidean Jordan algebra is $\bm{I}_n$.

We focus on the vectorized positive semidefinite cone in this paper.
We define the linear mapping $\svec\colon \mathbb{S}^n\to\mathbb{R}^{T_n}$ as
\begin{equation*}
\svec(\bm{X}) \coloneqq (X_{1,1},\sqrt{2}X_{1,2},X_{2,2},\dots,\sqrt{2}X_{1,n},\dots,\sqrt{2}X_{n-1,n},X_{n,n})
\end{equation*}
for each $\bm{X}\in\mathbb{S}^n$.
Then, the mapping $\svec(\cdot)$ is an isometry between $\mathbb{S}^n$ and $\mathbb{R}^{T_n}$; i.e., $\langle \bm{X},\bm{Y}\rangle = \svec(\bm{X})^\top\svec(\bm{Y})$ holds for all $\bm{X},\bm{Y}\in\mathbb{S}^n$.
Let $\smat(\cdot)$ denote the inverse mapping of $\svec(\cdot)$.
If we define $\mathbb{E}\coloneqq \mathbb{R}^{T_n}$, $\bm{x}\circ \bm{y} \coloneqq \svec(\smat(\bm{x})\mathbin{\lozenge}\smat(\bm{y}))$, and $\bm{x}\bullet \bm{y} \coloneqq \bm{x}^\top\bm{y}$ for $\bm{x},\bm{y}\in\mathbb{E}$, then $(\mathbb{E},\circ,\bullet)$ is a Euclidean Jordan algebra, and the associated symmetric cone $\mathbb{E}_+$ is $\svec(\mathbb{S}_+^n)$.
We also refer to the vectorized positive semidefinite cone $\svec(\mathbb{S}_+^n)$ as the positive semidefinite cone hereafter.
\end{example}

\section{Generalized Doubly Nonnegative Cone}\label{sec:DNN}
For a class of closed cones $\mathbb{K}$, we consider a GDNN cone $\DNN(\mathbb{K})$ over $\mathbb{K}$.
The following requirements are natural.
First, to regard $\mathcal{DNN}(\mathbb{K})$ as a \emph{generalization of the DNN cone}, we require $\DNN(\mathbb{R}_+^n) = \DNN^n$.
Second, we require that $\DNN(\mathbb{K})$ be a cone.
Third, we require $\CP(\mathbb{K}) \subseteq \DNN(\mathbb{K}) \subseteq \mathbb{S}_+^n$ to provide tighter relaxation than SDP for GCPP as the DNN cone does.

In this section, as described in Section~\ref{sec:intro}, we focus on the fact that the dual cone of the zeroth level of Parrilo's hierarchy is the DNN cone, thereby defining the dual cone of the zeroth level of two inner-approximation hierarchies~\cite{NN20XX,ZVP2006} for the GCOP cone as a GDNN cone.
In addition, we introduce the GDNN cone proposed by Burer and Dong~\cite{BD2012}.

\subsection{ZVP-type generalized doubly nonnegative cone}\label{subsec:ZVP_GDNN}
We first introduce the inner-approximation hierarchy proposed by Zuluaga, Vera, and Pe\~{n}a~\cite{ZVP2006} for $\COP(\mathbb{K})$, where $\mathbb{K}$ is a (closed) pointed semialgebraic convex cone.
Specifically, we suppose that $\mathbb{K}$ is expressed as
\begin{equation}
\{\bm{x}\in \mathbb{R}^n\mid \phi_i(\bm{x})\ge 0\ (i = 1,\dots,q)\}, \label{eq:K_semialg}
\end{equation}
where $\phi_i\in H^{n,m_i}$ for $i=1,\dots,q$, and that $\mathbb{K}$ is a pointed convex cone.
Then, there exists some $\bm{a}\in \mathbb{R}^n\setminus\{\bm{0}\}$ such that the following holds (see \cite[Section~6]{ZVP2006}):
\begin{align}
&\mathbb{K}\subseteq \{\bm{x}\in\mathbb{R}^n\mid \bm{a}^\top\bm{x}\ge 0\}, \label{eq:exposed_face_1}\\
&\mathbb{K} \cap \{\bm{x}\in\mathbb{R}^n\mid \bm{a}^\top\bm{x} = 0\} = \{\bm{0}\}. \label{eq:exposed_face_2}
\end{align}
In fact, the set of $\bm{a}$ satisfying \eqref{eq:exposed_face_1} and \eqref{eq:exposed_face_2} is exactly $\setint(\mathbb{K}^*)$, which is nonempty because $\mathbb{K}$ is a pointed closed convex cone.
From \eqref{eq:exposed_face_1}, the geometric property of $\mathbb{K}$ does not change even if we add the nonnegative constraint $\bm{a}^\top\bm{x}\ge 0$ with $\bm{a}\in \setint(\mathbb{K}^*)$ into \eqref{eq:K_semialg}.
Thus, we may assume that there is some $i$ such that $\phi_i(\bm{x}) = \bm{a}^\top\bm{x}$.
Such a redundant constraint is necessary to construct the inner-approximation hierarchy.
Note that a different choice of $\bm{a}$ may yield a different hierarchy even if the geometric property of $\mathbb{K}$ does not change. Namely, the hierarchy depends on the algebraic description of $\mathbb{K}$. Therefore, we may use the notation $(\mathbb{K};\bm{a})$ to emphasize the choice of $\bm{a}$.
Under the assumption on $\mathbb{K}$, the inner-approximation hierarchy presented by Zuluaga, Vera, and Pe\~{n}a for $\COP(\mathbb{K})$ is described as follows.

\begin{theorem}[{\cite[Proposition~17]{ZVP2006}}]\label{thm:ZVP_inner_approx_COP}
Let
\begin{equation}
E^{n,m}(\mathbb{K}) \coloneqq \conv\left\{\psi^2\prod_{j=1}^k\phi_{i_j}\relmiddle|
\begin{array}{l}
k\in \mathbb{N},\ m-\sum_{j=1}^km_{i_j}\in \mathbb{N}\text{ is even},\\
\psi\in H^{n,(m-\sum_{j=1}^km_{i_j})/2},\\
i_j\in \{1,\dots,q\}\ (j=1,\dots,k)
\end{array}
\right\}\label{eq:EnmK}
\end{equation}
and $\mathcal{K}_{{\rm ZVP},r}(\mathbb{K}) \coloneqq \{\bm{A}\in\mathbb{S}^n\mid (\bm{a}^\top\bm{x})^r\bm{x}^\top\bm{A}\bm{x}\in E^{n,r+2}(\mathbb{K})\}$ for each $r\in\mathbb{N}$.
The sequence $\{\mathcal{K}_{{\rm ZVP},r}(\mathbb{K})\}_r$ then satisfies the following conditions:
\begin{enumerate}[(i)]
\item $\mathcal{K}_{{\rm ZVP},r}(\mathbb{K}) \subseteq \mathcal{K}_{{\rm ZVP},r+1}(\mathbb{K}) \subseteq \COP(\mathbb{K})$ for all $r \in \mathbb{N}$.
\item $\setint\COP(\mathbb{K}) \subseteq \bigcup_{r=0}^{\infty}\mathcal{K}_{{\rm ZVP},r}(\mathbb{K})$.
\end{enumerate}
\end{theorem}

More generally, Zuluaga, Vera, and Pe\~{n}a proposed inner-approximation hierarchies for the cone of copositive homogeneous polynomials or tensors over a given cone, and we considered a similar approach~\cite{NN20XX}.
However, we restrict our scope to matrices in the present work in accordance with Burer and Dong~\cite{BD2012}.

Using this hierarchy, we define the ZVP-type GDNN cone as follows:

\begin{definition}
Let $\mathbb{K}$ be a pointed semialgebraic convex cone in $\mathbb{R}^n$. Then, we define $\mathcal{K}_{{\rm ZVP},0}(\mathbb{K})^*$ as the ZVP-type GDNN cone over $\mathbb{K}$, written as $\DNN_{\rm ZVP}(\mathbb{K})$.
\end{definition}

Before checking whether the ZVP-type GDNN cone satisfies the requirements mentioned at the beginning of Section~\ref{sec:DNN}, we provide an explicit expression of $\DNN_{\rm ZVP}(\mathbb{K})$.
Note that because the polynomial $\psi^2\prod_{j=1}^k\phi_{i_j}$ appearing in \eqref{eq:EnmK} has a degree of $m$, homogeneous polynomials of degree exceeding $2$ appearing in \eqref{eq:K_semialg} are ignored in the construction of $E^{n,2}(\mathbb{K})$ and therefore $\mathcal{K}_{{\rm ZVP},0}(\mathbb{K})$.
Thus, assuming that the degrees of homogeneous polynomials appearing in \eqref{eq:K_semialg} are at most $2$, we now express $\mathbb{K}$ as
\begin{equation}
\{\bm{x}\in\mathbb{R}^n\mid \bm{x}^\top \bm{Q}_i\bm{x} \ge 0\ (i=1,\dots,q_2),\ \bm{a}_i^\top\bm{x}\ge 0\ (i=1,\dots,q_1)\},\label{eq:semialg_set_quad}
\end{equation}
where $\bm{Q}_i\in \mathbb{S}^n$ for $i = 1,\dots,q_2$ and $\bm{a}_i\in\mathbb{R}^n$ for $i = 1,\dots,q_1$.
Recall that $\mathcal{K}_{{\rm ZVP},0}(\mathbb{K}) = \{\bm{A}\in \mathbb{S}^{n}\mid \bm{x}^\top \bm{A}\bm{x}\in E^{n,2}(\mathbb{K})\}$.

\begin{lemma}\label{lem:K_ZVP}
Suppose that $\mathbb{K}$ is expressed as \eqref{eq:semialg_set_quad}. Then,
\begin{equation*}
\mathcal{K}_{{\rm ZVP},0}(\mathbb{K}) = \mathbb{S}_+^n + \sum_{i=1}^{q_2}\mathbb{R}_+\bm{Q}_i + \sum_{1\le i< j\le q_1}\mathbb{R}_+ \bm{A}^{ij},
\end{equation*}
where $\bm{A}^{ij} \coloneqq (\bm{a}_i\bm{a}_j^\top + \bm{a}_j\bm{a}_i^\top)/2$.
\end{lemma}

\begin{proof}
Let $\bm{A} \in \mathcal{K}_{{\rm ZVP},0}(\mathbb{K})$. Note that we can replace the convex hull operator in \eqref{eq:EnmK} with the conical hull operator because the inner set of the right-hand side of \eqref{eq:EnmK} is a cone containing zero. Then, because $\bm{x}^\top \bm{A}\bm{x}\in E^{n,2}(\mathbb{K})$, there exist $\bm{\psi}_t\in\mathbb{R}^n$, $\psi_{2(t)}\in\mathbb{R}$, $i_{2(t)}\in\{1,\dots,q_2\}$, $\psi_{1(t)}\in\mathbb{R}$, and $i_{1(t)}$, $j_{1(t)} \in \{1,\dots,q_1\}$ such that
\begin{align}
\bm{x}^\top \bm{A} \bm{x} &=
\sum_{t}(\bm{\psi}_t^\top\bm{x})^2 + \sum_{t}\psi_{2(t)}^2(\bm{x}^\top \bm{Q}_{i_{2(t)}}\bm{x}) +  \sum_{t}\psi_{1(t)}^2(\bm{a}_{i_{1(t)}}^\top\bm{x}) (\bm{a}_{j_{1(t)}}^\top\bm{x})\label{eq:K_ZVP_1}\\
&= \bm{x}^\top\left(\sum_{t}\bm{\psi}_t\bm{\psi}_t^\top + \sum_{t\in T_1}\psi_{1(t)}^2\bm{a}_{i_{1(t)}}\bm{a}_{i_{1(t)}}^\top + \sum_{t}\psi_{2(t)}^2\bm{Q}_{i_{2(t)}} +
\sum_{t\in T_2}\psi_{1(t)}^2\bm{A}^{i_{1(t)}j_{1(t)}}\right)\bm{x},\label{eq:K_ZVP_2}
\end{align}
where $T_1 \coloneqq \{t\mid i_{1(t)} = j_{1(t)}\}$ and $T_2 \coloneqq \{t\mid i_{1(t)} \neq j_{1(t)}\}$.
The first summand in \eqref{eq:K_ZVP_1} corresponds to the case of $k$ appearing in \eqref{eq:EnmK} equal to $0$; in the second, $k=1$; and in the third, $k=2$.
Since \eqref{eq:K_ZVP_2} holds for all $\bm{x}$, it follows that
\begin{equation*}
\bm{A} = \sum_{t}\bm{\psi}_t\bm{\psi}_t^\top + \sum_{t\in T_1}\psi_{1(t)}^2\bm{a}_{i_{1(t)}}\bm{a}_{i_{1(t)}}^\top + \sum_{t}\psi_{2(t)}^2\bm{Q}_{i_{2(t)}} +
\sum_{t\in T_2}\psi_{1(t)}^2\bm{A}^{i_{1(t)}j_{1(t)}},
\end{equation*}
which belongs to $\mathbb{S}_+^n + \sum_{i=1}^{q_2}\mathbb{R}_+\bm{Q}_i + \sum_{1\le i < j\le q_1}\mathbb{R}_+ \bm{A}^{ij}$.

Conversely, let $\bm{A} \in \mathbb{S}_+^n + \sum_{i=1}^{q_2}\mathbb{R}_+\bm{Q}_i + \sum_{1\le i< j\le q_1}\mathbb{R}_+\bm{A}^{ij}$.
The matrix $\bm{A}$ can then be written as $\bm{P} + \sum_{i=1}^{q_2}\psi_{i} \bm{Q}_i + \sum_{1\le i< j\le q_1}\psi_{ij}\bm{A}^{ij} $, where $\bm{P} \in \mathbb{S}_+^n$, $\psi_i \ge 0$ for $i = 1,\dots,q_2$, and $\psi_{ij}\ge 0$ for $1\le i< j\le q_1$. Because $\bm{P} \in \mathbb{S}_+^n$, there exist $\bm{p}_1,\dots,\bm{p}_n\in\mathbb{R}^n$ such that $\bm{P} =\sum_{i=1}^n\bm{p}_i\bm{p}_i^\top$. Therefore,
\begin{equation*}
\bm{x}^\top \bm{A}\bm{x} = \sum_{i=1}^n(\bm{p}_i^\top\bm{x})^2 + \sum_{i=1}^{q_2}({\textstyle \sqrt{\psi_i}})^2\bm{x}^\top \bm{Q}_i\bm{x} + \sum_{1\le i< j\le q_1}({\textstyle \sqrt{ \psi_{ij}}})^2(\bm{a}_i^\top\bm{x})(\bm{a}_j^\top\bm{x}) ,
\end{equation*}
which belongs to $\mathcal{K}_{{\rm ZVP},0}(\mathbb{K})$.
\end{proof}

\begin{proposition}\label{prop:DNN_ZVP}
Suppose that $\mathbb{K}$ is expressed as \eqref{eq:semialg_set_quad}. Then,
\begin{equation*}
\DNN_{\rm ZVP}(\mathbb{K}) = \mathbb{S}_+^n \cap \bigcap_{i=1}^{q_2}\{\bm{X}\in\mathbb{S}^n\mid \langle \bm{Q}_i,\bm{X}\rangle \ge 0\} \cap \bigcap_{1\le i< j\le q_1}\{\bm{X}\in \mathbb{S}^n\mid \langle \bm{A}^{ij},\bm{X}\rangle \ge 0\}.
\end{equation*}
\end{proposition}

\begin{proof}
Note that for $\bm{S}\in\mathbb{S}^n$, we have $(\mathbb{R}_+ \bm{S})^* = \{\bm{X}\in \mathbb{S}^n\mid \langle \bm{S},\bm{X}\rangle \ge 0\}$. In addition, for any closed convex cones $\mathcal{K}_1,\dots,\mathcal{K}_k$, $(\sum_{i=1}^k\mathcal{K}_i)^* = \bigcap_{i=1}^k\mathcal{K}_i^*$ holds. Using these facts, we obtain the desired result by taking the dual in Lemma~\ref{lem:K_ZVP}.
\end{proof}

It is reasonable to call $\DNN_{\rm ZVP}(\mathbb{K})$ a GDNN cone.
First, when $\mathbb{K} = \mathbb{R}_+^n= \{\bm{x}\in\mathbb{R}^n \mid \bm{e}_i^\top\bm{x} \ge 0\ (i = 1,\dots,n)\}$ and $\bm{a}$ is given as $\bm{1}_n\in \setint(\mathbb{K}^*)$, the hierarchy $\{\mathcal{K}_{{\rm ZVP},r}(\mathbb{R}_+^n;\bm{1}_n)\}_r$ agrees with Parrilo's hierarchy (see \cite[Remark~2]{ZVP2006} for its reason), the $r$th level of which is described as
\begin{equation}
\left\{\bm{A}\in\mathbb{S}^n \relmiddle| (\bm{x}^\top\bm{x})^r\sum_{i,j=1}^nA_{i,j}x_i^2x_j^2 \in \Sigma^{n,2(r+2)}\right\}. \label{eq:Parrilo}
\end{equation}
This implies that $\DNN_{\rm ZVP}(\mathbb{R}_+^n;\bm{1}_n) = \DNN^n$.
Second, $\DNN_{\rm ZVP}(\mathbb{K})$ is a cone, as it is obtained by dualizing $\mathcal{K}_{{\rm ZVP},0}(\mathbb{K})$.
Third, it follows from Theorem~\ref{thm:ZVP_inner_approx_COP} that $\CP(\mathbb{K}) \subseteq \DNN_{\rm ZVP}(\mathbb{K})$ and from Proposition~\ref{prop:DNN_ZVP} that $\DNN_{\rm ZVP}(\mathbb{K}) \subseteq \mathbb{S}_+^n$.

Note that unlike $\mathbb{S}_+^n +\mathcal{N}^n$, which is closed, $\mathcal{K}_{{\rm ZVP},0}(\mathbb{K})$ is not closed in general.
See Appendix~\ref{apdx:example_ZVP_not_closed} for an example of $\mathbb{K}$ such that $\mathcal{K}_{{\rm ZVP},0}(\mathbb{K})$ is not closed.
Therefore, the dual cone of $\DNN_{\rm ZVP}(\mathbb{K})$ is not $\mathcal{K}_{{\rm ZVP},0}(\mathbb{K})$ itself, but instead its closure.
However, at least when $\mathbb{K}$ is a direct product of a nonnegative orthant and second-order or positive semidefinite cones, $\mathcal{K}_{{\rm ZVP},0}(\mathbb{K})$ is closed (see Sections~\ref{subsec:soc} and~\ref{subsec:sdc}).

\subsection{NN-type generalized doubly nonnegative cone}\label{subsec:NN_GDNN}
Let $(\mathbb{E},\circ,\bullet)$ be a Euclidean Jordan algebra where $\mathbb{E}$ is the space $\mathbb{R}^n$.
As noted above, we previously proposed an inner-approximation hierarchy for the GCOP cone over the symmetric cone $\mathbb{E}_+$~\cite{NN20XX}.

\begin{theorem}[{\cite[Theorem~3.3]{NN20XX}}]\label{thm:inner_approx_COP}
Let
\begin{equation*}
\mathcal{K}_{{\rm NN},r}(\mathbb{E}_+) \coloneqq \{\bm{A}\in\mathbb{S}^n \mid (\bm{x}^\top\bm{x})^r(\bm{x}\circ\bm{x})^\top\bm{A}(\bm{x}\circ\bm{x})\in\Sigma^{n,2(r+2)}\}
\end{equation*}
for each $r\in\mathbb{N}$.
Then, $\mathcal{K}_{{\rm NN},r}(\mathbb{E}_+)$ is a closed convex cone for each $r\in\mathbb{N}$, and the sequence $\{\mathcal{K}_{{\rm NN},r}(\mathbb{E}_+)\}_r$ satisfies the following conditions:
\begin{enumerate}[(i)]
\item $\mathcal{K}_{{\rm NN},r}(\mathbb{E}_+) \subseteq \mathcal{K}_{{\rm NN},r+1}(\mathbb{E}_+) \subseteq \COP(\mathbb{E}_+)$ for all $r \in \mathbb{N}$.
\item $\setint\COP(\mathbb{E}_+) \subseteq \bigcup_{r=0}^{\infty}\mathcal{K}_{{\rm NN},r}(\mathbb{E}_+)$.
\end{enumerate}
\end{theorem}

Using this hierarchy, we define the NN-type GDNN cone:

\begin{definition}
Let $(\mathbb{E},\circ,\bullet)$ be a Euclidean Jordan algebra where $\mathbb{E}$ is the space $\mathbb{R}^n$.
We then define $\mathcal{K}_{{\rm NN},0}(\mathbb{E}_+)^*$ as the NN-type GDNN cone over $\mathbb{E}_+$, written as $\DNN_{\rm NN}(\mathbb{E}_+)$.
\end{definition}

The set $\DNN_{\rm NN}(\mathbb{E}_+)$ may be considered a GDNN cone.
When the Euclidean Jordan algebra is taken as in Example~\ref{ex:nno}, the associated symmetric cone $\mathbb{E}_+$ is the nonnegative orthant $\mathbb{R}_+^n$ and the hierarchy $\{\mathcal{K}_{{\rm NN},r}(\mathbb{R}_+^n)\}_r$ agrees with Parrilo's hierarchy~\eqref{eq:Parrilo}.
This implies that $\DNN_{\rm NN}(\mathbb{R}_+^n) = \DNN^n$.
Moreover, $\DNN_{\rm NN}(\mathbb{E}_+)$ is a cone that includes $\CP(\mathbb{E}_+)$ for the same reason as does the ZVP-type GDNN cone.
Finally, for any $\bm{A} \in \mathbb{S}_+^n$, there exists $\bm{U}\in \mathbb{S}^n$ such that $\bm{A}$ can be decomposed into $\bm{U}^2$.
For each $\bm{x}\in\mathbb{E}$, because each element of $\bm{x}\circ \bm{x}$ is a quadratic homogeneous polynomial, we can write $\bm{x}\circ \bm{x}$ as $(\phi_1(\bm{x}),\dots,\phi_n(\bm{x}))$, where $\phi_i(\bm{x})\in H^{n,2}$ for $i = 1,\dots,n$.
Then, we have
\begin{equation*}
(\bm{x}\circ\bm{x})^\top \bm{A} (\bm{x}\circ\bm{x}) = \|\bm{U}(\bm{x}\circ\bm{x})\|^2
= \sum_{i=1}^n\left(\sum_{j=1}^nU_{i,j}\phi_j(\bm{x})\right)^2 \in \Sigma^{n,4}.
\end{equation*}
Therefore, $\mathbb{S}_+^n \subseteq \mathcal{K}_{{\rm NN},0}(\mathbb{E}_+)$, and so $\DNN_{\rm NN}(\mathbb{E}_+) \subseteq \mathbb{S}_+^n$ hold.

We should note, however, that a fully explicit expression of $\DNN_{\rm NN}(\mathbb{E}_+)$ has been unknown in the case of general symmetric cones.
According to the result presented in~\cite{NN20XX}, $\DNN_{\rm NN}(\mathbb{E}_+)$ is written as the \textit{closure} of a set defined by a positive semidefinite constraint.
We demonstrate that the closure operator can be removed in the case wherein $\mathbb{E}_+$ is a direct product of a nonnegative orthant and second-order or positive semidefinite cones.
Namely, in such cases, $\DNN_{\rm NN}(\mathbb{E}_+)$ is fully captured by a positive semidefinite constraint.
Note in addition that because $\mathcal{K}_{{\rm NN},0}(\mathbb{E}_+)$ is always closed as mentioned in Theorem~\ref{thm:inner_approx_COP}, the dual of $\DNN_{\rm NN}(\mathbb{E}_+)$ is precisely $\mathcal{K}_{{\rm NN},0}(\mathbb{E}_+)$.

\subsection{BD-type generalized doubly nonnegative cone}
Above, we propose two GDNN cones based on the inner-approximation hierarchies for the GCOP cone.
Independently, Burer and Dong~\cite{BD2012} proposed another GDNN cone over a closed convex cone without exploiting inner-approximation hierarchies.

\begin{definition}[{\cite[Section~4]{BD2012}}]\label{def:BD_GDNN}
For a closed convex cone $\mathbb{K}$ in $\mathbb{R}^n$, the BD-type GDNN cone $\DNN_{\rm BD}(\mathbb{K})$ over $\mathbb{K}$ is defined by $\mathbb{S}_+^n \cap \mathcal{N}(\mathbb{K})$, where $\mathcal{N}(\mathbb{K})\coloneqq \{\bm{X}\in\mathbb{S}^n\mid \bm{X}\bm{s}\in\mathbb{K}\ \text{for all $\bm{s}\in \Ext(\mathbb{K}^*)$}\}$.
\end{definition}

The set $\DNN_{\rm BD}(\mathbb{K})$ can also be considered a GDNN cone based on the requirements listed at the beginning of Section~\ref{sec:DNN}.
In fact, by the definition of $\DNN_{\rm BD}(\mathbb{K})$, it comprises a cone included in the positive semidefinite cone.
In addition, as discussed in \cite{BD2012}, $\DNN_{\rm BD}(\mathbb{R}_+^n) = \DNN^n$ holds and $\DNN_{\rm BD}(\mathbb{K})$ includes $\CP(\mathbb{K})$.

\section{Analysis on Three Generalized Doubly Nonnegative Cones}
\label{sec:relationship}
In the previous section, we provided two GDNN cones $\DNN_{\rm ZVP}(\mathbb{K})$ and $\DNN_{\rm NN}(\mathbb{K})$ and introduced another, $\DNN_{\rm BD}(\mathbb{K})$, proposed by Burer and Dong~\cite{BD2012}.
Note that for a general closed cone $\mathbb{K}$, some GDNN cones may not be defined, in which case we cannot compare the three GDNN cones.
That is, $\DNN_{\rm ZVP}(\mathbb{K})$ is defined for pointed semialgebraic convex cones, whereas $\DNN_{\rm NN}(\mathbb{K})$ is defined for symmetric cones.
Therefore, in this section, we consider two special cases to define all three cones.
In particular, we consider the case in which $\mathbb{K}$ is a direct product of a nonnegative orthant and second-order cones.
In addition, as a preliminary investigation, we also consider the case in which $\mathbb{K}$ is a direct product of a nonnegative orthant and positive semidefinite cones.
Recall that the sequence $\{\mathcal{K}_{{\rm ZVP},r}(\mathbb{R}_+^n;\bm{1}_n)\}_r$ agrees with Parrilo's hierarchy and the vector $\bm{1}_n$ is the identity element of the Euclidean Jordan algebra associated with $\mathbb{R}_+^n$, as shown in Example~\ref{ex:nno}.
Accordingly, we select the vector $\bm{a}\in \setint(\mathbb{K}^*)$ in the ZVP-type GDNN cone $\DNN_{\rm ZVP}(\mathbb{K};\bm{a})$ as the identity element of a Euclidean Jordan algebra associated with $\mathbb{K}$, and no longer specify it.

For the case in which $\mathbb{K}$ is a direct product of a nonnegative orthant and \emph{one} second-order cone, i.e., $\mathbb{K} = \mathbb{R}_+^{n_1} \times \mathbb{L}^{n_2}$, we the authors~\cite{NN20XX}, previously selected $\bm{a}\in \setint(\mathbb{K}^*)$ as in this paper, to construct a ZVP-type hierarchy $\{\mathcal{K}_{{\rm ZVP},r}(\mathbb{K};\bm{a})\}_r$.
Upon solving the approximation problems that resulted from substituting the GCOP cone with hierarchies such as ZVP- and NN-types, a \emph{numerical} comparison of these hierarchies was conducted.
In this study, we consider the case in which $\mathbb{K}$ is a direct product of a nonnegative orthant and \emph{multiple} second-order cones and investigate the inclusion relationship between the dual cone of the zeroth level of the ZVP- and NN-type hierarchies \emph{theoretically}.

\subsection{When $\mathbb{K}$ is a direct product of a nonnegative orthant and second-order cones}
\label{subsec:soc}
In this subsection, we consider the case in which $\mathbb{K}$ is a direct product of a nonnegative orthant and second-order cones---i.e., $\mathbb{K} = \mathbb{R}_+^{n_1}\times \prod_{h=2}^N\mathbb{L}^{n_h}$---and let $n \coloneqq \sum_{h=1}^Nn_h$.
For convenience, we reindex $(1,\dots,n)$ as $(11,\dots,1n_1,21,\dots,2n_2,\dots,N1,\dots,Nn_N)$; i.e., let $hi \coloneqq \sum_{k=1}^{h-1}n_k + i$ for $h = 1,\dots,N$ and $i = 1,\dots,n_h$. Moreover, we set $\mathcal{I}_h \coloneqq \{h1,\dots,hn_h\}$ for $h = 1,\dots,N$, $\mathcal{I}_h^- \coloneqq \mathcal{I}_h\setminus\{h1\}$ for $h = 2,\dots,N$, and $\mathcal{I}_{\ge 0}\coloneqq \mathcal{I}_1\cup \bigcup_{h=2}^N\{h1\}$.
Let $(\mathbb{E}_1,\circ_1,\bullet_1)$ be the Euclidean Jordan algebra associated with the nonnegative orthant $\mathbb{R}_+^{n_1}$ shown in Example~\ref{ex:nno}, and $(\mathbb{E}_h,\circ_h,\bullet_h)$ be the Euclidean Jordan algebra associated with the second-order cone $\mathbb{L}^{n_h}$ shown in Example~\ref{ex:soc} for $h = 2,\dots,N$.
Taking the direct product of the $N$ Euclidean Jordan algebras, we obtain the Euclidean Jordan algebra $(\mathbb{E},\circ,\bullet)$ with $\mathbb{E}_+ = \mathbb{K}$.
For each $\bm{x}\in\mathbb{E}$, the square $\bm{x}\circ\bm{x}$ is calculated as
\begin{equation*}
\bm{x}\circ\bm{x} = \left((x_I^2)_{I\in\mathcal{I}_1},\sum_{I\in\mathcal{I}_2}x_I^2,(2x_{21}x_I)_{I\in\mathcal{I}_2^-},\dots,\sum_{I\in\mathcal{I}_N}x_I^2,(2x_{N1}x_I)_{I\in\mathcal{I}_N^-}\right).
\end{equation*}
The identity element $\bm{e}$ of the Euclidean Jordan algebra $(\mathbb{E},\circ,\bullet)$ is the vector with $I$th $(I\in\mathcal{I}_{\ge 0})$ elements $1$ and all other elements $0$.

\subsubsection{Explicit expression of GDNN cones}
First, we provide a simpler explicit expression of the ZVP-type GDNN cone $\DNN_{\rm ZVP}(\mathbb{K})$.
Note that the cone $\mathbb{K}$ has the following semialgebraic representation:
\begin{equation}
\left\{\bm{x}\in\mathbb{R}^n \relmiddle|
\begin{array}{l}
\bm{e}_I^\top\bm{x} = x_I \ge 0\ (I\in\mathcal{I}_{\ge 0}),\\
\bm{e}^\top\bm{x} = \sum_{I\in\mathcal{I}_{\ge 0}}x_I \ge 0,\\
\bm{x}^\top\bm{J}_h\bm{x} = x_{h1}^2 - \sum_{I\in\mathcal{I}_h^-}x_I^2 \ge 0\ (h=2,\dots,N)
\end{array}\right\}, \label{eq:soc_semialgebraic}
\end{equation}
where $\bm{J}_{h}$ is the $n\times n$ matrix such that the $(h1,h1)$th element is $1$, the $(h2,h2),\dots,(hn_h,hn_h)$th elements are $-1$, and all other elements are $0$ for $h = 2,\dots,N$.
As mentioned in Section~\ref{subsec:ZVP_GDNN}, remember that the redundant constraint $\bm{e}^\top\bm{x} \ge 0$ is necessary to construct the ZVP-type hierarchy.
Consequently, Lemma~\ref{lem:K_ZVP} and Proposition~\ref{prop:DNN_ZVP} yield explicit expressions of $\mathcal{K}_{{\rm ZVP},0}(\mathbb{K})$ and $\DNN_{\rm ZVP}(\mathbb{K})$, respectively.
However, in this case, they can be expressed in a simpler way.

\begin{lemma}\label{lem:K_ZVP_SOC}
\begin{equation}
\mathcal{K}_{{\rm ZVP},0}(\mathbb{K})
= \mathbb{S}_+^n + \sum_{h=2}^N\mathbb{R}_+\bm{J}_h
+ \left\{\bm{N}\in\mathbb{S}^n\relmiddle|
N_{I,J} \begin{cases}
\ge 0& (I,J\in\mathcal{I}_{\ge 0},\ I\neq J)\\
= 0& (otherwise)
\end{cases}
\right\}.\label{eq:K_ZVP_SOC}
\end{equation}
\end{lemma}

Throughout this subsection, for simplicity, we express the last set in the right-hand side of \eqref{eq:K_ZVP_SOC} as $\mathcal{N}_{\ge 0}$.

\begin{proof}
Let $\bm{E}^{IJ} \coloneqq (\bm{e}_I\bm{e}_J^\top + \bm{e}_J\bm{e}_I^\top)/2$ and $\bm{E}_I \coloneqq (\bm{e}_I\bm{e}^\top + \bm{e}\bm{e}_I^\top)/2$.
By Lemma~\ref{lem:K_ZVP}, we have
\begin{equation*}
\mathcal{K}_{{\rm ZVP},0}(\mathbb{K}) = \mathbb{S}_+^n + \sum_{h=2}^N\mathbb{R}_+\bm{J}_h + \sum_{\substack{I< J\\I,J\in\mathcal{I}_{\ge 0}}}\mathbb{R}_+\bm{E}^{IJ} + \sum_{I\in\mathcal{I}_{\ge 0}}\mathbb{R}_+\bm{E}_I.
\end{equation*}
Therefore, it is sufficient to show that
\begin{equation*}
\mathbb{S}_+^n + \sum_{\substack{I< J\\I,J\in\mathcal{I}_{\ge 0}}}\mathbb{R}_+\bm{E}^{IJ} + \sum_{I\in\mathcal{I}_{\ge 0}}\mathbb{R}_+\bm{E}_I = \mathbb{S}_+^n + \mathcal{N}_{\ge 0}.
\end{equation*}
For each $I,J\in \mathcal{I}_{\ge 0}$ with $I < J$, since $\bm{E}^{IJ}$ is the matrix with $(I,J)$th and $(J,I)$th elements equal to $1/2$ and other elements $0$, we have $\bm{E}^{IJ} \in \mathcal{N}_{\ge 0}$.
In addition, for each $I \in \mathcal{I}_{\ge 0}$, we have
\begin{equation*}
\bm{E}_I = \bm{e}_I\bm{e}_I^\top + \sum_{\substack{J\in\mathcal{I}_{\ge 0}\\J\neq I}}\bm{E}^{IJ} \in \mathbb{S}_+^n + \mathcal{N}_{\ge 0}.
\end{equation*}
These imply the ``$\subseteq$" component.
Conversely, let $\bm{N}\in \mathcal{N}_{\ge 0}$.
Then,
\begin{equation*}
\bm{N} = \sum_{\substack{I< J\\I,J\in\mathcal{I}_{\ge 0}}}2N_{I,J}\bm{E}^{IJ}\in \sum_{\substack{I< J\\I,J\in\mathcal{I}_{\ge 0}}}\mathbb{R}_+ \bm{E}^{IJ},
\end{equation*}
which implies the ``$\supseteq$'' component, thereby completing the proof.
\end{proof}

Taking the dual in Lemma~\ref{lem:K_ZVP_SOC}, we have a simpler explicit expression of $\DNN_{\rm ZVP}(\mathbb{K})$.

\begin{proposition}\label{prop:GDNN_ZVP_SOC}
For $\mathbb{K} = \mathbb{R}_+^{n_1}\times \prod_{h=2}^N\mathbb{L}^{n_h}$, it follows that
\begin{equation*}
\DNN_{\rm ZVP}(\mathbb{K}) =
\mathbb{S}_+^n \cap \bigcap_{h=2}^N\{\bm{X}\in\mathbb{S}^n\mid \langle\bm{J}_h,\bm{X}\rangle\ge 0\} \cap\{\bm{N}\in\mathbb{S}^n\mid N_{I,J}\ge 0\ (I,J\in\mathcal{I}_{\ge 0},\ I\neq J)\}.
\end{equation*}
\end{proposition}

Next, we provide an explicit expression of the NN-type GDNN cone.
For $\bm{y} = (y_{\bm{\delta}})_{\bm{\delta} \in \mathbb{I}_{=4}^n} \in \mathbb{R}^{\mathbb{I}_{=4}^n}$, let $\bm{C}_0(\bm{y})\in\mathbb{S}^n$ be the matrix with the $(I,J)$th element:\footnote{
Although we only define the upper triangular blocks of $\bm{C}_0(\bm{y})$ in \eqref{def:C0}, the lower triangular blocks are further defined so that $\bm{C}_0(\bm{y})$ is a symmetric matrix.}
\begin{subequations}
\begin{alignat}{2}
&y_{2(\bm{e}_I+\bm{e}_J)} &\quad& (I,J\in\mathcal{I}_1), \label{def:eq:C1}\\
&\sum_{K\in\mathcal{I}_h}y_{2(\bm{e}_I+\bm{e}_K)} &\quad& (I\in\mathcal{I}_1,\ J = h1,\ h = 2,\dots,N), \label{def:eq:C2}\\
&2y_{2\bm{e}_I+\bm{e}_{h1}+\bm{e}_J} &\quad& (I\in\mathcal{I}_1,\ J\in\mathcal{I}_h^-,\ h = 2,\dots,N), \nonumber\\
&\sum_{K\in\mathcal{I}_g}\sum_{L\in\mathcal{I}_h} y_{2(\bm{e}_K+\bm{e}_L)} &\quad& (I=g1,\ J=h1,\ g,h=2,\dots,N),\label{def:eq:C8}\\
&\sum_{K\in\mathcal{I}_g} 2y_{2\bm{e}_K+\bm{e}_{h1}+\bm{e}_J} &\quad& (I=g1,\ J\in\mathcal{I}_h^-,\ g,h=2,\dots,N), \nonumber\\
&4y_{\bm{e}_{g1}+\bm{e}_I+\bm{e}_{h1}+\bm{e}_J} &\quad& (I\in\mathcal{I}_g^-,\ J\in\mathcal{I}_h^-,\ g,h=2,\dots,N).\nonumber
\end{alignat} \label{def:C0}
\end{subequations}

If we define $\mathcal{C}_0 \coloneqq \{\bm{C}_0(\bm{y})\mid \bm{y}\in \mathcal{M}^{n,4}\}$, we can show that $\mathcal{K}_{{\rm NN},0}(\mathbb{K}) = \mathcal{C}_0^*$ by the same argument as in \cite[Proposition~3.6]{NN20XX}.
Since $\bm{C}_0(\bm{y})$ is linear with respect to $\bm{y}$ and $\mathcal{M}^{n,4}$ is a convex cone, it follows that $\mathcal{C}_0$ is a convex cone.
Therefore, $\DNN_{\rm NN}(\mathbb{K}) = \cl\mathcal{C}_0$.
We now demonstrate that $\mathcal{C}_0$ is closed, and the closure operator can be removed.

\begin{proposition}\label{prop:C_is_closed}
The set $\mathcal{C}_0$ is closed.
\end{proposition}

\begin{proof}
Let $\{\bm{y}^{(k)}\}_k \subseteq \mathcal{M}^{n,4}$ and suppose that $\bm{C}_0(\bm{y}^{(k)})$ converges to some $\bm{C}_0^*$ in the limit $k\to\infty$. It follows from the definition of $\mathcal{M}^{n,4}$ that $\bm{M}^{n,2}(\bm{y}^{(k)})\in\mathbb{S}_+^{\mathbb{I}_{=2}^n}$, and in particular, the diagonal elements of $\bm{M}^{n,2}(\bm{y}^{(k)})$ are nonnegative. Hence, we see that
\begin{equation}
y_{2\bm{\gamma}}^{(k)} \ge 0\quad\text{for all $\bm{\gamma}\in \mathbb{I}_{=2}^n$}. \label{eq:nonneg_y}
\end{equation}
Since $\bm{C}_0(\bm{y}^{(k)}) \to \bm{C}_0^*$, each element of $\bm{C}_0(\bm{y}^{(k)})$ is bounded. Specifically,
\begin{enumerate}[(i)]
\item $\{y_{2(\bm{e}_I + \bm{e}_J)}^{(k)}\}_k$ is bounded for all $I,J\in\mathcal{I}_1$ (see \eqref{def:eq:C1}). \label{enum:1}

\item For each $h = 2,\dots,N$, $\{\sum_{J\in\mathcal{I}_h}y_{2(\bm{e}_I+\bm{e}_J)}^{(k)}\}$ is bounded for all $I\in\mathcal{I}_1$ (see \eqref{def:eq:C2}). Therefore, we see from \eqref{eq:nonneg_y} that $\{y_{2(\bm{e}_I+\bm{e}_J)}^{(k)}\}_k$ is bounded for all $I\in\mathcal{I}_1$ and $J\in\mathcal{I}_h$. \label{enum:3}

\item For each $g,h = 2,\dots,N$, $\{\sum_{I\in\mathcal{I}_g}\sum_{J\in\mathcal{I}_h} y_{2(\bm{e}_I+\bm{e}_J)}^{(k)}\}_k$ is bounded (see \eqref{def:eq:C8}). Therefore, $\{y_{2(\bm{e}_I+\bm{e}_J)}^{(k)}\}_k$ is bounded for all $I\in\mathcal{I}_g$ and $J\in\mathcal{I}_h$. \label{enum:5}
\end{enumerate}
Because $\mathbb{I}_{=2}^n = \{\bm{e}_I + \bm{e}_J \mid 1\le I\le J\le n\}$, it follows from \eqref{enum:1} to \eqref{enum:5} that $\{y_{2\bm{\gamma}}^{(k)}\}_k$ is bounded for all $\bm{\gamma} \in \mathbb{I}_{=2}^n$, which implies the boundedness of the diagonal elements of $\bm{M}^{n,2}(\bm{y}^{(k)})$.
Combining boundedness with the positive semidefiniteness of $\bm{M}^{n,2}(\bm{y}^{(k)})$ yields the boundedness of each element in $\bm{M}^{n,2}(\bm{y}^{(k)})$. Therefore, $\{\bm{y}^{(k)}\}_k$ is also bounded, as $y_{\bm{\delta}}^{(k)}$ is an element of $\bm{M}^{n,2}(\bm{y}^{(k)})$ for each $\bm{\delta} \in \mathbb{I}_{=4}^n$. Therefore, there exists a convergent subsequence $\{\bm{y}^{(k_r)}\}_r$ with a limit of $\bm{y}^*$. Then, because $\mathcal{M}^{n,4}$ is closed and $\bm{y}^{(k_r)} \in \mathcal{M}^{n,4}$, we have $\bm{y}^* \in \mathcal{M}^{n,4}$. Since $\bm{C}_0(\bm{y})$ is continuous with respect to $\bm{y}$, we obtain $\bm{C}_0(\bm{y}^{(k_r)}) \to \bm{C}_0(\bm{y}^*) = \bm{C}_0^*$.
Thus, $\mathcal{C}_0$ is closed.
\end{proof}

\begin{corollary}\label{cor:DNN_NN}
For $\mathbb{K} = \mathbb{R}_+^{n_1}\times \prod_{h=2}^N\mathbb{L}^{n_h}$, it follows that $\DNN_{\rm NN}(\mathbb{K}) = \mathcal{C}_0$.
\end{corollary}

\begin{proof}
It is clear from Proposition~\ref{prop:C_is_closed}.
\end{proof}

\begin{remark}\label{rem:DNN_NN}
For $\mathbb{K} = \mathbb{R}_+^{n_1}\times \prod_{h=2}^N\mathbb{L}^{n_h}$, in the same manner as \cite[Proposition~3.6]{NN20XX} and Corollary~\ref{cor:DNN_NN}, we can derive an explicit semidefinite representation of the dual cone of $\mathcal{K}_{{\rm NN},r}(\mathbb{K})$ for general cases of $r$.
The discussion can be extended to the case of tensors.
\end{remark}

\subsubsection{Inclusion relationship between GDNN cones}
The following subsubsection discusses the inclusion relationship between ZVP-, NN-, and BD-type GDNN cones.

To begin with, we show that $\DNN_{\rm NN}(\mathbb{K})$ is strictly included in $\DNN_{\rm ZVP}(\mathbb{K})$.
The proof requires the closedness of $\mathcal{K}_{{\rm ZVP},0}(\mathbb{K})$, which does not hold in the general case.
To prove closedness, we prepare an additional lemma that claims a sufficient condition that the Minkowski sum of pointed closed convex cones is closed.
Because it is a slight modification of \cite[Corollary~9.1.3]{Rockafellar1970}, the proof is omitted.

\begin{lemma}\label{lem:minlovski_closed}
Let $\mathcal{K}_i \subseteq \mathbb{S}^n\ (i = 1,\dots,m)$ be pointed closed convex cones.
Then, $\sum_{i=1}^m\mathcal{K}_i$ is closed if the following condition holds: for any $\bm{X}_i\in\mathcal{K}_i\ (i = 1,\dots,m)$, if $\sum_{i=1}^m\bm{X}_i = \bm{O}$, then $\bm{X}_i = \bm{O}$ for all $i = 1,\dots,m$.
\end{lemma}

\begin{lemma}\label{lem:KZVP_closed}
For $\mathbb{K} = \mathbb{R}_+^{n_1}\times \prod_{h=2}^N\mathbb{L}^{n_h}$, the convex cone $\mathcal{K}_{{\rm ZVP},0}(\mathbb{K})$ is closed.
\end{lemma}

\begin{proof}
Let $\bm{P}\in\mathbb{S}_+^n$, $t_2,\dots,t_N\ge 0$, and $\bm{N}\in\mathcal{N}_{\ge 0}$, and suppose that
\begin{equation*}
\bm{A} \coloneqq \bm{P} + \sum_{h=2}^Nt_h\bm{J}_h + \bm{N} = \bm{O}.
\end{equation*}
For every $h = 2,\dots,N$, because $A_{h1,h1} = P_{h1,h1} + t_h = 0$ and each term is nonnegative, we have $P_{h1,h1} = t_h = 0$.
Given that $t_h = 0$ for all $h = 2,\dots,N$, for each $I \in \mathcal{I}_1\cup\bigcup_{h=2}^N\mathcal{I}_h^-$, we have $A_{I,I} = P_{I,I} = 0$.
Therefore, the diagonal elements of $\bm{P}$ are all 0 and $\bm{P} = \bm{O}$, as $\bm{P}\in\mathbb{S}_+^n$. Finally, we obtain $\bm{A} = \bm{N} = \bm{O}$.
Obviously, $\mathbb{S}_+^n$ and $\mathbb{R}_+\bm{J}_h\ (h=2,\dots,N)$ are pointed closed convex cones.
In addition, because $\mathcal{N}_{\ge 0}\subseteq \mathcal{N}^n$, $\mathcal{N}_{\ge 0}$ is also a pointed closed convex cone. $\mathcal{K}_{{\rm ZVP},0}(\mathbb{K})$ is therefore closed by Lemma~\ref{lem:minlovski_closed}.
\end{proof}

\begin{lemma}\label{lem:DNNd_inclusion_ZVP_NN}
For $\mathbb{K} = \mathbb{R}_+^{n_1}\times \prod_{h=2}^N\mathbb{L}^{n_h}$, $\mathcal{K}_{{\rm ZVP},0}(\mathbb{K}) \subseteq \mathcal{K}_{{\rm NN},0}(\mathbb{K})$ holds.
Moreover, the inclusion holds strictly if $n_1 \ge 1$ and $n_2 \ge 2$.
\end{lemma}

\begin{proof}
Let
\begin{equation}
\bm{A} = \bm{P} + \sum_{h = 2}^Nt_h\bm{J}_{h} + \bm{N} \in \mathcal{K}_{{\rm ZVP},0}(\mathbb{K}), \label{eq:mat_in_K_ZVP_SOC}
\end{equation}
where $\bm{P}\in \mathbb{S}_+^n$, $t_2,\dots,t_N\ge 0$, and $\bm{N}\in \mathcal{N}_{\ge 0}$.
First, let $\bm{U}$ be a real $n\times n$ symmetric matrix such that $\bm{P} = \bm{U}^2$; then, $(\bm{x}\circ\bm{x})^\top \bm{P}(\bm{x}\circ\bm{x}) = \|\bm{U}(\bm{x}\circ\bm{x})\|^2 \in \Sigma^{n,4}$. Second, for each $h = 2,\dots,N$,
\begin{equation*}
(\bm{x}\circ\bm{x})^\top(t_h\bm{J}_h)(\bm{x}\circ\bm{x}) =  t_h\left(x_{h1}^2 - \sum_{I\in\mathcal{I}_h^-}x_I^2\right)^2 \in \Sigma^{n,4}.
\end{equation*}
Finally,
\begin{align*}
(\bm{x}\circ\bm{x})^\top \bm{N}(\bm{x}\circ\bm{x}) &= \sum_{I,J\in\mathcal{I}_1}N_{I,J}(x_Ix_J)^2 + 2\sum_{h=2}^N\sum_{I\in\mathcal{I}_1}\sum_{J\in\mathcal{I}_h}N_{I,h1}(x_Ix_J)^2\\
&\quad+ \sum_{g,h=2}^N\sum_{I\in\mathcal{I}_g}\sum_{J\in\mathcal{I}_h}N_{g1,h1}(x_Ix_J)^2 \in \Sigma^{n,4}.
\end{align*}
Therefore, $(\bm{x}\circ\bm{x})^\top \bm{A}(\bm{x}\circ\bm{x}) \in \Sigma^{n,4}$, which means that $\bm{A}\in \mathcal{K}_{{\rm NN},0}(\mathbb{K})$.

The following example shows that the inclusion holds strictly. Suppose that $n_1\ge 1$ and $n_2\ge 2$. Let $\bm{A}\in\mathbb{S}^n$ be a matrix with the $(11,21)$th, $(11,22)$th, $(21,11)$th, and $(22,11)$th elements $1$ and all other elements $0$.
Since
\begin{align*}
(\bm{x}\circ\bm{x})^\top \bm{A}(\bm{x}\circ\bm{x}) &= 2x_{11}^2\sum_{I\in\mathcal{I}_2}x_I^2 + 4x_{11}^2x_{21}x_{22}\\
&= 2x_{11}^2\left\{(x_{21}+x_{22})^2 + \sum_{I\in\mathcal{I}_2^-\setminus\{22\}}x_I^2\right\} \in \Sigma^{n,4},
\end{align*}
it follows that $\bm{A}\in \mathcal{K}_{{\rm NN},0}(\mathbb{K})$.

On the other hand, we assume that $\bm{A}\in\mathcal{K}_{{\rm ZVP},0}(\mathbb{K})$ and express $\bm{A}$ as \eqref{eq:mat_in_K_ZVP_SOC}.
Then, we have $0 = A_{11,11} = P_{11,11}$.
Combining $P_{11,11} = 0$ with $\bm{P}\in\mathbb{S}_+^n$ yields $P_{11,22} = 0$. Therefore, $A_{11,22}$ must be 0, which contradicts the definition of $\bm{A}$.
Hence, $\bm{A} \not\in \mathcal{K}_{{\rm ZVP},0}(\mathbb{K})$.
\end{proof}

\begin{theorem}\label{thm:DNN_inclusion_NN_ZVP}
For $\mathbb{K} = \mathbb{R}_+^{n_1}\times \prod_{h=2}^N\mathbb{L}^{n_h}$, $\DNN_{\rm NN}(\mathbb{K}) \subseteq \DNN_{\rm ZVP}(\mathbb{K})$ holds.
Moreover, the inclusion holds strictly if $n_1 \ge 1$ and $n_2 \ge 2$.
\end{theorem}

\begin{proof}
By taking the dual in Lemma~\ref{lem:DNNd_inclusion_ZVP_NN}, we have $\DNN_{\rm NN}(\mathbb{K}) \subseteq \DNN_{\rm ZVP}(\mathbb{K})$. We next show by contradiction that the strict inclusion $\DNN_{\rm NN}(\mathbb{K}) \subsetneq \DNN_{\rm ZVP}(\mathbb{K})$ holds under the assumption of $n_1 \ge 1$ and $n_2 \ge 2$. Assume that $\DNN_{\rm NN}(\mathbb{K}) = \DNN_{\rm ZVP}(\mathbb{K})$ holds.
By taking its dual, we have $\DNN_{\rm ZVP}(\mathbb{K})^* = \DNN_{\rm NN}(\mathbb{K})^*$.
Remember that the double dual of a closed convex cone agrees with itself (see Theorem~\mbox{2.1.\Rnum{3}}).
Since $\mathcal{K}_{{\rm ZVP},0}(\mathbb{K})$ and $\mathcal{K}_{{\rm NN},0}(\mathbb{K})$ are both closed convex cones (see Lemma~\ref{lem:KZVP_closed} and Theorem~\ref{thm:inner_approx_COP}, respectively), it follows from the definition of the ZVP- and NN-type GDNN cones that $\mathcal{K}_{{\rm ZVP},0}(\mathbb{K}) = \mathcal{K}_{{\rm NN},0}(\mathbb{K})$.
This contradicts the assumption and Lemma~\ref{lem:DNNd_inclusion_ZVP_NN}.
\end{proof}

We will demonstrate that, in general, there is no inclusion relationship between $\DNN_{\rm ZVP}(\mathbb{K})$ and $\DNN_{\rm BD}(\mathbb{K})$.
Before presenting such examples, we provide a more concise expression of $\mathcal{N}(\mathbb{K})$ appearing in $\DNN_{\rm BD}(\mathbb{K})$ for the case where $\mathbb{K}$ is a direct product of a nonnegative orthant and second-order cones.
It can be easily seen that the set $\mathfrak{J}$, defined as
\begin{equation*}
\mathfrak{J} \coloneqq \bigcup_{I\in\mathcal{I}_1}\{(\bm{e}_I,\bm{0}_{n_2},\dots,\bm{0}_{n_N})\} \cup \bigcup_{h=2}^N\{(\bm{0}_{n_1},\bm{0}_{n_2},\dots,1/2,\bm{v}/2,\dots,\bm{0}_{n_N}) \mid \bm{v}\in S^{n_h-2}\},
\end{equation*}
satisfies $\Ext(\mathbb{K}) = \{\alpha\bm{s} \mid \alpha > 0,\bm{s}\in\mathfrak{J}\}$.
Since the vector $\bm{X}\bm{s}$ is linear with respect to $\bm{s}$ and $\mathbb{K}$ is self-dual, we can equivalently represent $\mathcal{N}(\mathbb{K})$ as $\{\bm{X}\in\mathbb{S}^n\mid \bm{X}\bm{s}\in\mathbb{K}\ \text{for all $\bm{s}\in \mathfrak{J}$}\}$.

\begin{example}
[a matrix that is in $\DNN_{\rm ZVP}(\mathbb{K})$ but not in $\DNN_{\rm BD}(\mathbb{K})$]\label{ex:inZVP_notinBD} Suppose that $n_1 \ge 1$ and $n_2\ge 2$. Let
\begin{equation*}
\bm{A} = \left(
\begin{array}{ccccc}
\Diag(n_2-1,\bm{O}_{n_1-1}) & \vline &\begin{matrix}0 & \bm{1}_{n_2-1}^\top\\\bm{0}_{n_1-1} & \bm{O}_{(n_1-1)\times(n_2-1)} \end{matrix}&\vline & \bm{O}\\
\hline
\begin{matrix}0 & \bm{0}_{n_1-1}^\top\\\bm{1}_{n_2-1} & \bm{O}_{(n_2-1)\times(n_1-1)}\end{matrix} &\vline &\Diag(n_2-1,\bm{I}_{n_2-1}) &\vline &\bm{O}\\
\hline
\bm{O} &\vline& \bm{O} &\vline & \bm{O}
\end{array}\right).
\end{equation*}
We use Proposition~\ref{prop:GDNN_ZVP_SOC} to show that $\bm{A}\in\DNN_{\rm ZVP}(\mathbb{K})$. We can easily check that
\begin{equation*}
\bm{A} \in \bigcap_{h=2}^N\{\bm{X}\in\mathbb{S}^n\mid \langle\bm{J}_h,\bm{X}\rangle\ge 0\} \cap
\{\bm{N}\in\mathbb{S}^n\mid N_{I,J}\ge 0\ (I,J\in\mathcal{I}_{\ge 0},\ I\neq J)\}.
\end{equation*}
Therefore, it is sufficient to show that $\bm{A}\in \mathbb{S}_+^n$. $\bm{A}\in \mathbb{S}_+^n$ if and only if
\begin{equation}
\Diag\left[\bm{O}_{\sum_{h=1}^Nn_h-n_2-1},n_2-1,\begin{pmatrix}
n_2-1 & \bm{1}_{n_2-1}^\top\\
\bm{1}_{n_2-1} & \bm{I}_{n_2-1}
\end{pmatrix}\right] \in \mathbb{S}_+^n, \label{eq:ex_inZVP_notinNN_equiv}
\end{equation}
which is obtained by rearranging some rows and columns of $\bm{A}$.
Moreover, \eqref{eq:ex_inZVP_notinNN_equiv} is equivalent to
\begin{equation*}
\begin{pmatrix}
n_2-1&\bm{1}_{n_2-1}^\top\\
\bm{1}_{n_2-1} & \bm{I}_{n_2-1}
\end{pmatrix} \in \mathbb{S}_+^{n_2},
\end{equation*}
which is true because $(n_2 - 1) - \bm{1}_{n_2-1}^\top \bm{I}_{n_2-1}\bm{1}_{n_2-1} = 0$. (Here, we apply the Schur complement lemma~\cite[Lemma~4.2.1]{BN2001}.)
Therefore, we have $\bm{A}\in \DNN_{\rm ZVP}(\mathbb{K})$.

Next, we let
\begin{equation*}
\bm{s} = \left(\bm{0}_{n_1},\frac{1}{2},-\frac{\bm{1}_{n_2-1}}{2\sqrt{n_2-1}},\bm{0}\right)\in\mathfrak{J}.
\end{equation*}
Then, since $(\bm{A}\bm{s})_1 = -\sqrt{n_2-1}/2 < 0$, we see that $\bm{A}\bm{s} \not\in \mathbb{K}$.
(The first component of an element in $\mathbb{K}$ must be nonnegative.)
Thus, $\bm{A}\not\in \DNN_{\rm BD}(\mathbb{K})$.
\end{example}

\begin{example}
[a matrix that is in $\DNN_{\rm BD}(\mathbb{K})$ but not in $\DNN_{\rm ZVP}(\mathbb{K})$]\label{ex:inBD_notinZVP} Suppose that $n_2\ge 3$. Let
\begin{equation*}
\bm{A} = \Diag\left(\bm{O}_{n_1},1,\frac{\bm{I}_{n_2-1}}{\sqrt{n_2-1}},\bm{O}\right).
\end{equation*}
On the one hand, $\bm{A}\in \DNN_{\rm BD}(\mathbb{K})$.
Indeed, let $\bm{s}\in \mathfrak{J}$.
If $\bm{s} = (\bm{0}_{n_1},1/2,\bm{v}/2,\bm{0})$ for some $\bm{v}\in S^{n_2-2}$, then
\begin{equation*}
\bm{A}\bm{s} = \left(\bm{0}_{n_1}, \frac{1}{2}, \frac{\bm{v}}{2\sqrt{n_2-1}},\bm{0}\right).
\end{equation*}
Since
\begin{equation*}
\left(\frac{1}{2}\right)^2 - \left\|\frac{\bm{v}}{2\sqrt{n_2-1}}\right\|^2 = \frac{1}{4}\left(1-\frac{1}{n_2-1}\right) \ge 0,
\end{equation*}
$(\frac{1}{2}, \frac{\bm{v}}{2\sqrt{n_2-1}}) \in \mathbb{L}^{n_2}$, and therefore we get $\bm{A}\bm{s}\in\mathbb{K}$. Otherwise, $\bm{A}\bm{s} = \bm{0}\in \mathbb{K}$.

On the other hand, $\bm{A}\not\in \DNN_{\rm ZVP}(\mathbb{K})$ because $\langle\bm{J}_2,\bm{A}\rangle = 1 -\sqrt{n_2-1} < 0$.
\end{example}

We consider the special case of $\mathbb{K} = \mathbb{L}^n$.
As a consequence of the dual of the S-lemma~\cite{PT2007} (see also \cite[Theorem~1]{SZ2003}), it follows that
\begin{equation*}
\CP(\mathbb{L}^n) = \mathbb{S}_+^n \cap \{\bm{X}\in\mathbb{S}^n\mid \langle \Diag(1,-\bm{I}_{n-1}),\bm{X}\rangle\ge 0\}.
\end{equation*}
Therefore, $\CP(\mathbb{L}^n)$ itself is semidefinite representable.
Surprisingly, from Proposition~\ref{prop:GDNN_ZVP_SOC} and Theorem~\ref{thm:DNN_inclusion_NN_ZVP}, we see that $\DNN_{\rm ZVP}(\mathbb{L}^n)$ and $\DNN_{\rm NN}(\mathbb{L}^n)$ agree with $\CP(\mathbb{L}^n)$.
However, Example~\ref{ex:inBD_notinZVP} implies that $\DNN_{\rm BD}(\mathbb{L}^n)$ includes $\CP(\mathbb{L}^n)$ strictly if $n\ge 3$.
Including the case of $n \le 2$, we derive the necessary and sufficient condition that $\CP(\mathbb{L}^n)$ agrees with each GDNN cone over $\mathbb{L}^n$.

\begin{proposition}\label{prop:CP(Ln)}
$\CP(\mathbb{L}^n) = \DNN_{\rm NN}(\mathbb{L}^n) = \DNN_{\rm ZVP}(\mathbb{L}^n) \subseteq \DNN_{\rm BD}(\mathbb{L}^n)$ holds.
Furthermore, the above inclusion of ``$\subseteq$'' holds strictly if and only if $n\ge 3$.
\end{proposition}

\begin{proof}
It is sufficient to show that $\DNN_{\rm BD}(\mathbb{L}^n) \subseteq \CP(\mathbb{L}^n)$ in the case of $n \le 2$.
The equality holds when $n = 1$ because $\CP(\mathbb{L}^1) = \mathbb{S}_+^1$ and $\DNN_{\rm BD}(\mathbb{L}^1)$ is sandwiched between $\CP(\mathbb{L}^1)$ and $\mathbb{S}_+^1$.

Let us consider the case of $n = 2$.
For $\bm{X} \in \DNN_{\rm BD}(\mathbb{L}^2)$, $\bm{X}$ is positive semidefinite by the definition of $\DNN_{\rm BD}(\mathbb{L}^2)$.
In addition, since the vector $(1,1)^\top$ generates an extreme ray of $\mathbb{L}^2$, it follows that $\bm{X}(1,1)^\top = (X_{1,1}+X_{1,2},X_{1,2}+X_{2,2})^\top \in \mathbb{L}^2$, i.e., $X_{1,1}+X_{1,2} \ge |X_{1,2}+X_{2,2}|$.
This implies that $X_{1,1} - X_{2,2}$, which is equal to $\langle \Diag(1,-\bm{I}_{1}),\bm{X}\rangle$, is greater than or equal to $0$.
Thus, we have $\bm{X} \in \CP(\mathbb{L}^2)$.
\end{proof}

Finally, we discuss the inclusion relationship between $\DNN_{\rm NN}(\mathbb{K})$ and $\DNN_{\rm BD}(\mathbb{K})$. Because $\DNN_{\rm NN}(\mathbb{K}) \subseteq \DNN_{\rm ZVP}(\mathbb{K})$ (Theorem~\ref{thm:DNN_inclusion_NN_ZVP}), the matrix presented in Example~\ref{ex:inBD_notinZVP} is also in $\DNN_{\rm BD}(\mathbb{K})$, but not in $\DNN_{\rm NN}(\mathbb{K})$. Therefore, we are interested in whether there exists a matrix that is in $\DNN_{\rm NN}(\mathbb{K})$ but not in $\DNN_{\rm BD}(\mathbb{K})$. If no such matrix exists, $\DNN_{\rm NN}(\mathbb{K})$ is included in $\DNN_{\rm BD}(\mathbb{K})$. Although we were not able to determine this theoretically, the results of our numerical experiment imply that $\DNN_{\rm NN}(\mathbb{K})$ is included in $\DNN_{\rm BD}(\mathbb{K})$.
For details, see Section~\ref{subsec:misocp}.

\subsection{When $\mathbb{K}$ is a direct product of a nonnegative orthant and positive semidefinite cones}
\label{subsec:sdc}
In this subsection, we consider the case where $\mathbb{K}$ is a direct product of a nonnegative orthant and positive semidefinite cones; i.e., $\mathbb{K} = \mathbb{R}_+^{n_1}\times \prod_{h=2}^N\svec(\mathbb{S}_+^{n_h})$, and let $n \coloneqq n_1 + \sum_{h=2}^NT_{n_h}$. For convenience, we reindex $(1,\dots,n)$ as
\begin{equation*}
(11,\dots,1n_1,211,212,222,\dots,21n_2,\dots,2n_2n_2,\dots,N11,\dots,Nn_Nn_N),
\end{equation*}
i.e., $1i \coloneqq i$ for $i = 1,\dots,n_1$, and $hij \coloneqq n_1+\sum_{k=2}^{h-1}T_{n_k} + j(j-1)/2 + i$ for $h = 2,\dots,N$ and $1\le i\le j\le n_h$. For $h = 2,\dots,N$ and $1\le j < i\le n_h$, let $hij \coloneqq hji$.
Moreover, we set $\widetilde{\mathcal{I}}_h \coloneqq \{h11,h12,h22,\dots,hn_hn_h\}$ for $h = 2,\dots,N$ and
\begin{equation*}
\widetilde{\mathcal{I}}_{\ge 0} \coloneqq \{11,\dots,1n_1,211,222,233,\dots,2n_2n_2,\dots,N11,\dots,Nn_Nn_N\}.
\end{equation*}
Let $(\mathbb{E}_1,\circ_1,\bullet_1)$ be the Euclidean Jordan algebra associated with the nonnegative orthant $\mathbb{R}_+^{n_1}$ shown in Example~\ref{ex:nno}, and $(\mathbb{E}_h,\circ_h,\bullet_h)$ be the Euclidean Jordan algebra associated with the positive semidefinite cone $\svec(\mathbb{S}_+^{n_h})$ shown in Example~\ref{ex:sdc} for $h = 2,\dots,N$.
Taking the direct product of the $N$ Euclidean Jordan algebras, we obtain the Euclidean Jordan algebra $(\mathbb{E},\circ,\bullet)$ with $\mathbb{E}_+ = \mathbb{K}$.
For each $\bm{x}\in\mathbb{E}$, the square $\bm{x}\circ\bm{x}$ is calculated as
\begin{equation*}
(\bm{x}\circ\bm{x})_I = \begin{dcases}
x_I^2 & (I = 11,\dots,1n_1),\\
x_I^2 + \frac{1}{2}\sum_{k:k\neq i}x_{hki}^2 & (I = hii,\ h = 2,\dots,N,\ i = 1,\dots,n_h),\\
\begin{split}
&x_{hii}x_I + x_Ix_{hjj} \\
&\ + \frac{1}{\sqrt{2}}\sum_{k:k\neq i,j}x_{hik}x_{hjk}
\end{split} & (I = hij,\ h = 2,\dots,N,\ 1\le i < j\le n_h).
\end{dcases}
\end{equation*}
The identity element $\bm{e}$ of the Euclidean Jordan algebra $(\mathbb{E},\circ,\bullet)$ is the vector with $I$th $(I\in\widetilde{\mathcal{I}}_{\ge 0})$ elements $1$ and all other elements $0$.

As mentioned in Section~\ref{subsec:EJA}, the positive semidefinite cone $\svec(\mathbb{S}_+^m)$ is a symmetric cone.
In addition, $\svec(\mathbb{S}_+^m)$ is semialgebraic, as
\begin{equation}
\svec(\mathbb{S}_+^m) = \left\{\bm{x}\in\mathbb{R}^{T_m} \relmiddle|
\begin{array}{l}
\det\smat(\bm{x})_{\mathcal{I},\mathcal{I}} \ge 0 \text{ for all $\mathcal{I} \subseteq \{1,\dots,m\}$},\\
\bm{e}^\top\bm{x} = \sum_{I\in\widetilde{\mathcal{I}}_{\ge 0}}x_I \ge 0
\end{array}
\right\}. \label{eq:psd_semialgebraic}
\end{equation}
Therefore, we can define all three GDNN cones properly in this case.

Here, we investigate the inclusion relationship between the ZVP- and NN-type GDNN cones.
As in Section~\ref{subsec:soc}, we provide a simpler explicit expression of $\mathcal{K}_{{\rm ZVP},0}(\mathbb{K})$ and $\DNN_{\rm ZVP}(\mathbb{K})$:

\begin{lemma}\label{lem:K_ZVP_SDC}
Let $\bm{J}_h^{ij}$ be the $n\times n$ matrix with $(hii,hjj)$th and $(hjj,hii)$th elements $1$, $(hij,hij)$th element $-1$, and all other elements $0$.
Then,
\begin{equation}
\mathcal{K}_{{\rm ZVP},0}(\mathbb{K}) = \mathbb{S}_+^n + \sum_{h=2}^N\sum_{1\le i< j\le n_h}\mathbb{R}_+\bm{J}_h^{ij} + \left\{\bm{N}\in\mathbb{S}^n\relmiddle|
N_{I,J} \begin{cases}
\ge 0& (I,J\in\widetilde{\mathcal{I}}_{\ge 0},\ I\neq J)\\
= 0& (otherwise)
\end{cases}
\right\}. \label{eq:K_ZVP_SDC}
\end{equation}
\end{lemma}

For simplicity, we write the last set in the right-hand side of \eqref{eq:K_ZVP_SDC} as $\widetilde{\mathcal{N}}_{\ge 0}$.
Because Lemma~\ref{lem:K_ZVP_SDC} can be proven in the same way as Lemma~\ref{lem:K_ZVP_SOC}, the proof is omitted.
We can also show that $\mathcal{K}_{{\rm ZVP},0}(\mathbb{K})$ is closed using Lemma~\ref{lem:minlovski_closed}.
Moreover, taking the dual in Lemma~\ref{lem:K_ZVP_SDC}, we obtain a simpler explicit expression of $\DNN_{\rm ZVP}(\mathbb{K})$.

\begin{proposition}\label{prop:DNN_ZVP_SDC}
For $\mathbb{K} = \mathbb{R}_+^{n_1}\times \prod_{h=2}^N\svec(\mathbb{S}_+^{n_h})$, it follows that
\begin{multline*}
\DNN_{\rm ZVP}(\mathbb{K}) = \mathbb{S}_+^n \cap \bigcap_{h=2}^N\bigcap_{1\le i< j\le n_h}\{\bm{X}\in\mathbb{S}^n \mid \langle \bm{J}_h^{ij},\bm{X}\rangle \ge 0\}\\ \cap \{\bm{N}\in\mathbb{S}^n \mid
N_{I,J} \ge 0\ (I,J\in\widetilde{\mathcal{I}}_{\ge 0},\ I\neq J)\}. \label{eq:DNN_ZVP_SDC}
\end{multline*}
\end{proposition}

\begin{remark}\label{rem:DNN_NN_sd}
Similarly to Corollary~\ref{cor:DNN_NN}, we can obtain an explicit expression of $\DNN_{\rm NN}(\mathbb{K})$ without the closure operator.
In addition, this can be extended to the case of tensors.
\end{remark}

We show that the NN-type GDNN cone is strictly included in the ZVP-type cone as in the case that $\mathbb{K}$ involves second-order cones.

\begin{lemma}\label{lem:K_inclusion_ZVP_NN_SDC}
For $\mathbb{K} = \mathbb{R}_+^{n_1}\times \prod_{h=2}^N\svec(\mathbb{S}_+^{n_h})$, $\mathcal{K}_{{\rm ZVP},0}(\mathbb{K}) \subseteq \mathcal{K}_{{\rm NN},0}(\mathbb{K})$ holds. Moreover, the inclusion holds strictly if $n_1 \ge 1$ and $n_2 \ge 2$.
\end{lemma}
\begin{proof}
Let
\begin{equation}
\bm{A} = \bm{P} + \sum_{h=2}^N\sum_{1\le i< j\le n_h}t_h^{ij}\bm{J}_h^{ij} + \bm{N}, \label{eq:mat_in_K_ZVP_SDC}
\end{equation}
where $\bm{P}\in\mathbb{S}_+^n$, $t_h^{ij}\ge 0$, and $\bm{N}\in\widetilde{\mathcal{N}}_{\ge 0}$.
First, as the proof of Lemma~\ref{lem:DNNd_inclusion_ZVP_NN}, $(\bm{x}\circ\bm{x})^\top \bm{P}(\bm{x}\circ \bm{x})\in\Sigma^{n,4}$. Second, for $h = 2,\dots,N$ and $1\le i < j\le n_h$, because
\begin{align*}
&2(\bm{x}\circ\bm{x})^\top\bm{J}_h^{ij}(\bm{x}\circ\bm{x})\\
&\quad= 4\left(x_{hii}^2 + \frac{1}{2}\sum_{k:k\neq i}x_{hki}^2\right)\left(x_{hjj}^2  + \frac{1}{2}\sum_{k:k\neq j}x_{hkj}^2\right)\\
&\quad\qquad- 2\left(x_{hii}x_{hij} + x_{hij}x_{hjj} + \frac{1}{\sqrt{2}}\sum_{k:k\neq i,j}x_{hik}x_{hjk}\right)^2\\
&\quad= (x_{hij}^2-2x_{hii}x_{hjj})^2 + \sum_{k:k\neq i,j}(x_{hij}x_{hik}-\sqrt{2}x_{hii}x_{hjk})^2\\
&\quad\qquad + \sum_{k:k\neq i,j}(x_{hij}x_{hjk}-\sqrt{2}x_{hjj}x_{hik})^2 + \sum_{\substack{k<l\\k,l\neq i,j}}(x_{hik}x_{hjl}-x_{hil}x_{hjk})^2,
\end{align*}
we have $(\bm{x}\circ\bm{x})^\top\bm{J}_h^{ij}(\bm{x}\circ\bm{x})\in\Sigma^{n,4}$. Finally,
\begin{align}
&(\bm{x}\circ\bm{x})^\top \bm{N}(\bm{x}\circ\bm{x}) \nonumber\\
&\quad= \sum_{i,j=1}^{n_1}N_{1i,1j}x_{1i}^2x_{1j}^2 + 2\sum_{h=2}^N\sum_{i=1}^{n_1}\sum_{j=1}^{n_h}N_{1i,hjj}x_{1i}^2\left(x_{hjj}^2 + \frac{1}{2}\sum_{k:k\neq j}x_{hkj}^2\right)\nonumber\\
&\quad\qquad + \sum_{g,h=2}^N\sum_{i=1}^{n_g}\sum_{j=1}^{n_h} N_{gii,hjj}\left(x_{gii}^2 + \frac{1}{2}\sum_{k:k\neq i}x_{gki}^2\right)\left(x_{hjj}^2 + \frac{1}{2}\sum_{k:k\neq j}x_{hkj}^2\right). \label{eq:N_SDC}
\end{align}
As all variables that appear in \eqref{eq:N_SDC} are squared and each element of $\bm{N}$ is nonnegative, we see that $(\bm{x}\circ\bm{x})^\top \bm{N}(\bm{x}\circ\bm{x}) \in\Sigma^{n,4}$. Therefore, $(\bm{x}\circ\bm{x})^\top \bm{A}(\bm{x}\circ\bm{x})\in \Sigma^{n,4}$, which implies that $\bm{A}\in\mathcal{K}_{{\rm NN},0}(\mathbb{K})$.

The following example shows that the inclusion holds strictly.
Suppose that $n_1\ge 1$ and $n_2\ge 2$. Let $\epsilon$ be a sufficiently small positive value, $\bm{a} = (a_{ij})_{1\le i\le j\le n_2} \in \mathbb{R}^{T_{n_2}}$ be the vector such that
\begin{equation*}
a_{ij} = \begin{cases}
1 & (i = j),\\
\epsilon & (i < j),
\end{cases}
\end{equation*}
and $\bm{A}\in\mathbb{S}^n$ be the matrix with $(11,\widetilde{\mathcal{I}}_2)$th and $(\widetilde{\mathcal{I}}_2,11)$th elements $\bm{a}$ and all other elements $0$.
Then, some calculations yield the following:
\begin{dmath*}
(\bm{x}\circ\bm{x})^\top \bm{A}(\bm{x}\circ\bm{x}) = x_{11}^2\left[\epsilon\sum_{i<j}\{(x_{2ii} + x_{2ij})^2 + (x_{2ij} + x_{2jj})^2\} + \frac{\epsilon}{\sqrt{2}}\sum_{i<j}\sum_{k:k\neq i,j}(x_{2ik} + x_{2jk})^2  + \left\{2-(n_2-1)\epsilon\right\}\sum_{i=1}^{n_2}x_{2ii}^2 + [2-\{2 + \sqrt{2}(n_2-2)\}\epsilon]\sum_{i<j}x_{2ij}^2\right].
\end{dmath*}
Thus, $(\bm{x}\circ\bm{x})^\top \bm{A}(\bm{x}\circ\bm{x})$ can be represented as an SOS polynomial if
\begin{equation*}
\epsilon \le \min\left\{\frac{2}{n_2-1},\frac{2}{2 + \sqrt{2}(n_2-2)}\right\}.
\end{equation*}
Such a positive $\epsilon$ certainly exists.
Consequently, on the one hand, we have $\bm{A}\in \mathcal{K}_{{\rm NN},0}(\mathbb{K})$.

On the other hand, assume that $\bm{A}$ can be expressed as \eqref{eq:mat_in_K_ZVP_SDC}.
Then, we have $0 = A_{11,11} = P_{11,11}$.
Combining $P_{11,11} = 0$ with $\bm{P}\in\mathbb{S}_+^n$ yields $P_{11,2ij} = 0$ for all $i<j$. Therefore, for a pair $(i,j)$ with $i<j$, $A_{11,2ij}$ must be $0$. However, by the definition of $\bm{A}$, $A_{11,2ij} = \epsilon > 0$, which is a contradiction.
\end{proof}

Taking the dual in Lemma~\ref{lem:K_inclusion_ZVP_NN_SDC}, we obtain the following desired result:
\begin{theorem}\label{thm:DNN_inclusion_NN_ZVP_SDC}
For $\mathbb{K} = \mathbb{R}_+^{n_1}\times \prod_{h=2}^N\svec(\mathbb{S}_+^{n_h})$, $\DNN_{\rm NN}(\mathbb{K}) \subseteq \DNN_{\rm ZVP}(\mathbb{K})$ holds.
Moreover, the inclusion holds strictly if $n_1 \ge 1$ and $n_2 \ge 2$.
\end{theorem}

The theoretical inclusion relationship between ZVP- and BD-type GDNN cones and between NN- and BD-type GDNN cones remains unknown.
See also Section~\ref{subsec:max-cut} for the numerical comparison of the strength of relaxation by the three GDNN cones.

\section{Numerical Experiments}\label{sec:experiment}
We conducted a series of numerical experiments to investigate the strength of relaxation by the three GDNN cones.
We consider in Section~\ref{subsec:misocp}, the case where $\mathbb{K}$ is a direct product of a nonnegative orthant and second-order cones, and in Section~\ref{subsec:max-cut}, the case where $\mathbb{K}$ is a direct product of a nonnegative orthant and positive semidefinite cones.
We conducted all of the experiments with MATLAB (R2021b in Section~\ref{subsec:misocp} and R2022a in Section~\ref{subsec:max-cut}) on a computer with an Intel Core i5-8279U 2.40~\mbox{GHz} CPU and 16~\mbox{GB} of memory.
The versions of the YALMIP modeling language~\cite{Lofberg2004} and the MOSEK solver~\cite{MOSEK} we used in the following two experiments are 20210331 and 9.3.3, respectively.

\subsection{Mixed 0--1 second-order cone programming}\label{subsec:misocp}
We considered the following mixed 0--1 second-order cone programming problem with a linear objective function.
\begin{equation}
\begin{alignedat}{3}
&\minimize_{\bm{x}} && \quad\bm{c}^\top\bm{x}\\
&\subjectto && \quad 0\le x_1 \le 2,\\
&&& \quad 0\le x_i \le 1\ (i=2,\dots,n),\\
&&& \quad x_i\in\{0,1\}\ (i\in B\subseteq \{2,\dots,n\}),\\
&&&\quad \bm{x}\in\mathbb{L}^n.
\end{alignedat}\label{prob:MISOCP}
\end{equation}
As a consequence of Burer's findings, the second constraint of problem~\eqref{prob:MISOCP} ensures that we can equivalently transform problem~\eqref{prob:MISOCP} into a GCPP problem~\cite{Burer2012}.
Specifically, by introducing $2n$ nonnegative slack variables, we can convert problem~\eqref{prob:MISOCP} into the primal standard form of GCPP
\begin{equation}
\begin{alignedat}{3}
&\minimize_{\bm{X}} && \quad \langle \bm{C},\bm{X}\rangle \\
&\subjectto && \quad \langle \bm{A}_i, \bm{X}\rangle = b_i\ (i=1,\dots,m),\\
&&&\quad \bm{X}\in \CP(\mathbb{K})
\end{alignedat}\label{prob:GCPP}
\end{equation}
with $m = 4n+{|B|}+1$ and $\mathbb{K} = \mathbb{R}_+^{2n+1}\times \mathbb{L}^n$ for some symmetric matrices $\bm{A}_1,\dots,\bm{A}_m$, $\bm{C}$ and scalars $b_1,\dots,b_m$.
In this section, we solve the problem obtained by relaxing the GCPP problem with each GDNN cone and compare the results.

Reformulating problem~\eqref{prob:MISOCP} as a GCPP problem would be impractical because problem~\eqref{prob:MISOCP} can be solved directly and quickly with an existing mixed-integer second-order cone programming solver.
However, we decided to solve a mixed 0--1 second-order cone programming problem with a linear objective function rather than a nonlinear or nonconvex objective because we were able to obtain the exact optimal value of the problem with the solver and calculate the difference between the optimal value of the original problem and that of its relaxation problem with each GDNN cone.

We created instances of problem~\eqref{prob:MISOCP} as follows.
The number $n$ of variables was changed between $5$, $10$, $30$, and $50$.
The set $B$ of indices to determine binary variables was generated by selecting $0.4n$ elements from $\{2,\dots,n\}$ randomly.
All elements of $\bm{c}$ were independent and identically distributed, and each followed the standard normal distribution. For each $n$, we generated five instances varying the randomness of $B$ and $\bm{c}$.
For each instance, the optimal value of the relaxation problem with $\DNN_{\rm NN}(\mathbb{K})$ (written as NN in Table~\ref{tab:MISOCP_result}), $\DNN_{\rm ZVP}(\mathbb{K})$ (ZVP), and $\DNN_{\rm BD}(\mathbb{K})$ (BD) was measured.
For reference, we also solved the problem~\eqref{prob:MISOCP} itself (MISOCP), as well as the relaxation problem with the positive semidefinite cone (SDP). To improve numerical stability, $0.005\bm{I}$ was added to the coefficient matrix $\bm{C}$ in the standard form~\eqref{prob:GCPP} when we solved the SDP relaxation problems.
We solved a BD-type GDNNP problem as a semi-infinite conic programming problem by adopting an algorithm based on the explicit exchange method~\cite{OHF2012}. See Appendix~\ref{apdx:algorithm} for the algorithm.
We used the YALMIP modeling language and the MOSEK solver to solve the optimization problems.

\begin{table}[tbp]
\centering
\caption{Optimal values of problem~\eqref{prob:MISOCP} and its associated GDNNP and SDP relaxation problems}
\scalebox{0.9}{
\begin{tabular}{rrrrrrr}
\hline
    &     & \multicolumn{5}{c}{Optimal value} \\
$n$ & No. & MISOCP & NN    & ZVP    & BD     & SDP \\
\hline
$5$   & 1   & $-$4.01  & $-$4.01 & $-$4.01  & $-$4.01  & $-$30.56 \\
    & 2   & $-$3.18  & $-$3.18 & $-$3.18  & $-$3.18  & $-$4.93 \\
    & 3   & $-$0.28  & $-$0.28 & $-$0.31  & $-$0.28  & $-$16.86 \\
    & 4   & 0.00   & 0.00  & 0.00   & 0.00   & $-$106.50 \\
    & 5   & $-$1.47  & $-$1.47 & $-$1.47  & $-$1.47  & $-$163.54 \\
$10$  & 1   & $-$4.73  & OOM   & $-$4.74  & $-$4.77  & $-$42.56  \\
    & 2   & $-$5.43  & OOM   & $-$5.43  & $-$5.43  & $-$140.18 \\
    & 3   & $-$1.99  & OOM   & $-$1.99  & $-$1.99  & $-$13.20 \\
    & 4   & 0.00   & OOM   & $-$0.86  & 0.00   & $-$186.66 \\
    & 5   & $-$4.23  & OOM   & $-$4.23  & $-$4.23  & $-$53.90 \\
$30$  & 1   & $-$9.17  & OOM   & $-$9.17  & $-$9.62  & $-$328.08 \\
    & 2   & $-$9.42  & OOM   & $-$9.42  & $-$9.95  & $-$340.17 \\
    & 3   & $-$5.91  & OOM   & $-$5.91  & $-$6.00  & $-$120.07 \\
    & 4   & $-$3.05  & OOM   & $-$3.07  & $-$3.25  & $-$444.75 \\
    & 5   & $-$9.58  & OOM   & $-$9.58  & $-$10.14 & $-$362.77 \\
$50$  & 1   & $-$9.34  & OOM   & $-$9.34  & $-$10.93 & $-$489.81 \\
    & 2   & $-$9.70  & OOM   & $-$9.71  & $-$10.25 & $-$324.13 \\
    & 3   & $-$6.68  & OOM   & $-$6.69  & $-$7.16  & $-$293.58 \\
    & 4   & $-$3.30  & OOM   & $-$3.30  & $-$3.78  & $-$560.80 \\
    & 5   & $-$12.07 & OOM   & $-$12.10 & $-$13.29 & $-$750.69 \\
\hline
\end{tabular}
}\label{tab:MISOCP_result}
\end{table}

Table~\ref{tab:MISOCP_result} lists the optimal values of the problem~\eqref{prob:MISOCP} and its GDNNP and SDP relaxation problems, where ``OOM'' means that we could not solve the problem because of insufficient memory.

It may be observed from Table~\ref{tab:MISOCP_result} that the optimal values of the GDNNP relaxation problems were much better than those of the SDP relaxation problems and were close to those of the original problems.
This implies that GDNNP relaxation for GCPP provides much tighter bounds than SDP relaxation.
In particular, the NN- and ZVP-type GDNNP relaxations provided nearly optimal values for the original problems.
These results are consistent with those of our previous study mentioned in Section~\ref{sec:intro}, in which we approximated the GCOP cone with the dual cone of the NN- and ZVP-type GDNN cones, i.e., $\mathcal{K}_{{\rm NN},0}(\mathbb{K})$ and $\mathcal{K}_{{\rm ZVP},0}(\mathbb{K})$.

ZVP-type GDNNP relaxation exhibited better numerical properties than BD-type GDNNP relaxation.
As illustrated in Examples~\ref{ex:inZVP_notinBD} and \ref{ex:inBD_notinZVP}, the inclusion relationship between ZVP- and BD-type GDNN cones does not hold.
However, the optimal values of the ZVP-type GDNNP relaxation problems are better than those of the BD-type problems in most cases, especially when $n\ge 30$.

Given the theoretical inclusion relationship between NN- and ZVP-type GDNN cones shown in Theorem~\ref{thm:DNN_inclusion_NN_ZVP} and the superior numerical performance of the ZVP-type GDNN cone compared to the BD-type cone, it is reasonable to conjecture that the NN-type cone provides numerically tighter relaxation than the BD-type.
It may be observed from Table~\ref{tab:MISOCP_result} that the optimal values of the NN-type GDNNP relaxation problems were always not worse than those of the BD-type problems. Of note, the theoretical inclusion relationship between NN- and BD-type GDNN cones remains unknown.
We therefore presume that the NN-type GDNN cone is included in the BD-type GDNN cone.

\begin{conjecture}\label{conj:DNN_inclusion_NN_BD}
If $\mathbb{K}$ is a direct product of a nonnegative orthant and second-order cones, $\DNN_{\rm NN}(\mathbb{K}) \subseteq \DNN_{\rm BD}(\mathbb{K})$ holds.
\end{conjecture}

Note that this conjecture is true where $\mathbb{K}$ is a second-order cone by Proposition~\ref{prop:CP(Ln)}.
Thus far, we have demonstrated the inclusion relationship between the three GDNN cones in Theorem~\ref{thm:DNN_inclusion_NN_ZVP}, Examples~\ref{ex:inZVP_notinBD}, \ref{ex:inBD_notinZVP}, and Conjecture~\ref{conj:DNN_inclusion_NN_BD}.
These results are illustrated in Figure~\ref{fig:inclusion}.

\begin{figure}[tbp]
\centering
\includegraphics[width=.3\linewidth]{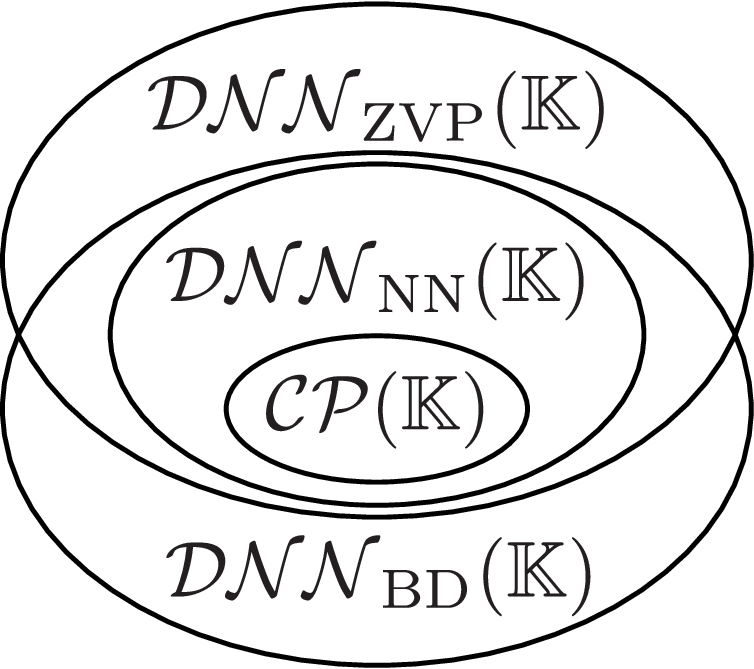}
\caption{Inclusion relationship between ZVP-, NN-, and BD-type generalized doubly nonnegative cones (including Conjecture~\ref{conj:DNN_inclusion_NN_BD})}
 \label{fig:inclusion}
\end{figure}

We have observed that the NN-type GDNN cone provided the tightest relaxation of the three GDNN cones from both theoretical and numerical perspectives.
However, we could not compute the NN-type GDNNP relaxation problems in the case of $n \ge 10$ due to insufficient memory.
\footnote{We tried to solve the NN-type GDNNP relaxation problems on another computer with 64 GB of memory, but the same issue persisted.}
We suspect that one of the reasons is that the NN-type GDNN cone is described by the positive semidefiniteness of the moment matrix.
The NN-type GDNNP relaxation problem of problem~\eqref{prob:GCPP} includes a positive semidefinite constraint of size $O(n^2)$, whereas the ZVP-type GDNNP relaxation problem only includes a positive semidefinite constraint of size $O(n)$.

\subsection{Max-cut problem}\label{subsec:max-cut}
Let $G = (V,E)$ be an undirected graph with nonnegative weights $(w_{ij})_{ij\in E}$, where $V \coloneqq \{1,\dots,n\}$ is a set of vertices, and $E \subseteq \{\{i,j\} \mid 1\le i < j\le n\}$ is a set of edges. We write $ij$ for $\{i,j\} \in E$.
The max-cut problem, a well-known NP-hard problem~\cite{Karp1972}, aims to find a subset $S \subseteq V$ that maximizes the sum of weights on edges connecting $S$ and $V\setminus S$.
For notational convenience, for $1\le i\le j \le n$ with $ij\not \in E$, we define $w_{ij} \coloneqq 0$.
Then, the max-cut problem can be formulated as the following binary quadratic programming problem:
\begin{equation}
\begin{alignedat}{3}
&\maximize_{\bm{x}} && \quad \frac{1}{4}\sum_{i,j=1}^nw_{ij}(1-x_ix_j)\\
&\subjectto && \quad x_i^2 = 1\ (i=1,\dots,n).
\end{alignedat}\label{prob:max-cut}
\end{equation}
Introducing a matrix variable $\widehat{\bm{X}} \coloneqq \bm{x}\bm{x}^\top$, up to a constant term and the sign of the optimal value, \eqref{prob:max-cut} is equivalent to an SDP problem with the rank constraint $\rank(\widehat{\bm{X}}) \le 1$.
Furthermore, by \cite[Theorem~5]{BMP2016}, we can equivalently transform the rank-constrained SDP problem into a GCPP problem in the form of \eqref{prob:GCPP} with $m = n^2 + 3n + 4$ and $\mathbb{K} = \mathbb{R}_+ \times \svec(\mathbb{S}_+^n)^3$.
As in Section~\ref{subsec:misocp}, to solve the problem obtained, we relax the GCPP problem with each GDNN cone and compare the results.

Instances of problem~\eqref{prob:max-cut} were created as follows.
The number $n$ of vertices was changed between $3$ and $5$. \footnote{Comparing the three GDNN cones for larger $n$ generated numerical issues. For $n = 10$, the computational time for solving BD-type GDNNP relaxation problems exceeded 14400~s, and due to insufficient memory, we were unable to solve NN-type GDNNP relaxation problems.}
For each $1 \le i < j\le n$, we independently generated an edge $ij$ with probability $1/2$.
All elements of $(w_{ij})_{ij\in E}$ were independent and identically distributed, and each followed the uniform distribution on the interval $(0,1)$.
For each $n$, we generated five instances by varying the randomness of $E$ and $(w_{ij})_{ij\in E}$.
For each instance, the optimal value of the relaxation problem with $\DNN_{\rm NN}(\mathbb{K})$ (written as NN in Table~\ref{tab:max-cut_result}), $\DNN_{\rm ZVP}(\mathbb{K})$ (ZVP), and $\DNN_{\rm BD}(\mathbb{K})$ (BD) was measured.
For reference, we also solved problem~\eqref{prob:max-cut} (MC) and the associated relaxation problems with the positive semidefinite cone $\mathbb{S}_+^{3T_n + 1}$ (SDP).
To improve numerical stability, $0.005\bm{I}$ was added to the coefficient matrix $\bm{C}$ in the standard form~\eqref{prob:GCPP} when solving SDP relaxation problems.
As in Section~\ref{subsec:misocp}, the BD-type GDNNP problem was solved by the algorithm shown in Appendix~\ref{apdx:algorithm}.
To solve NN-, ZVP-, and BD-type GDNNP and SDP relaxation problems, we used the YALMIP modeling language and the MOSEK solver. To solve \eqref{prob:max-cut}, we used the Gurobi solver~\cite{Gurobi} (version 10.0.3).

\begin{table}[]
\centering
\caption{Optimal values of problem~\eqref{prob:max-cut} and its associated GDNNP and SDP relaxation problems}
\scalebox{0.9}{
\begin{tabular}{rrrrrrr}
\hline
 & & \multicolumn{5}{c}{Optimal value}           \\
$n$ & No. & MC   & NN   & ZVP  & BD   & SDP \\
\hline
$3$   & 1   & 0.81 & 0.81 & 0.81 & 0.81 & 13.58    \\
    & 2   & 0.33 & 0.33 & 0.33 & 0.33 & 2.89     \\
    & 3   & 1.00 & 1.00 & 1.00 & 1.00 & 15.14    \\
    & 4   & 1.19 & 1.19 & 1.19 & 1.19 & 25.41    \\
    & 5   & 0.87 & 0.87 & 0.87 & 0.87 & 19.38    \\
$5$   & 1   & 2.10 & OOM  & 2.10 & 2.10 & 36.55    \\
    & 2   & 1.81 & OOM  & 2.24 & 1.81 & 26.67    \\
    & 3   & 3.04 & OOM  & 3.32 & 3.04 & 45.36    \\
    & 4   & 4.05 & OOM  & 4.30 & 4.05 & 91.34    \\
    & 5   & 1.94 & OOM  & 2.13 & 1.94 & 35.26 \\
\hline
\end{tabular}
}\label{tab:max-cut_result}
\end{table}

Table~\ref{tab:max-cut_result} lists the optimal values of problem~\eqref{prob:max-cut} and the accompanying GDNNP and SDP relaxation problems, where ``OOM'' means that we could not solve the problem because of insufficient memory.

GDNNP relaxation for GCPP provided significantly tighter bounds than SDP relaxation, a finding that aligns with the results presented in Section~\ref{subsec:misocp}.
However, unlike Section~\ref{subsec:misocp}, the ZVP-type GDNN cone provided worse relaxation than the BD-type cone.
This might be because the semialgebraic representation~\eqref{eq:psd_semialgebraic} of $\svec(\mathbb{S}_+^n)$ involves polynomials of degree exceeding $2$ when $n \ge 3$; whereas the semialgebraic representation~\eqref{eq:soc_semialgebraic} of the direct product of a nonnegative orthant and second-order cones only involves polynomials of degree at most $2$.
As mentioned in Section~\ref{subsec:ZVP_GDNN}, polynomials of degree exceeding $2$ appearing in a semialgebraic representation of $\mathbb{K}$ do not contribute to the construction of $\DNN_{\rm ZVP}(\mathbb{K})$.
Therefore, we can assume that the more the polynomials of degree exceeding $2$, the looser is the ZVP-type GDNNP relaxation.

\section{Conclusion}
\label{sec:conclusion}
In this study, we theoretically and numerically compared the strength of the relaxation of ZVP-, NN-, and BD-type GDNN cones over cones $\mathbb{K}$.
In particular, we considered two types of cones $\mathbb{K}$: a direct product of a nonnegative orthant and second-order cones and a direct product of a nonnegative orthant and positive semidefinite cones.
When $\mathbb{K}$ is a direct product of a nonnegative orthant and second-order cones, no theoretical inclusion relationship is obtained between ZVP- and BD-type GDNN cones, and the ZVP-type GDNN cone provides better relaxation numerically than BD-type cone. By contrast, when $\mathbb{K}$ is a direct product of a nonnegative orthant and positive semidefinite cones, the ZVP-type GDNN cone numerically provides worse relaxation than BD-type cone. In both cases, the NN-type GDNN cone is theoretically strictly included in the ZVP-type cone.
The results of our numerical experiments show that the three GDNN cones yield much tighter bounds for GCPP than the positive semidefinite cone, which suggests a promising avenue for further study.

As noted in Section~\ref{sec:intro}, we aimed to explore whether the ZVP-type $\{\mathcal{K}_{{\rm ZVP},r}(\mathbb{K})\}_r$ and NN-type hierarchies $\{\mathcal{K}_{{\rm NN},r}(\mathbb{K})\}_r$ are the same.
Lemmas~\ref{lem:DNNd_inclusion_ZVP_NN} and \ref{lem:K_inclusion_ZVP_NN_SDC} address this problem specifically for the case where $\mathbb{K}$ is a direct product of a nonnegative orthant and second-order or positive semidefinite cones.
These lemmas establish the strict inclusion relationship between the zeroth level of the two hierarchies.
However, whether the inclusion relationship between them also holds for a general $r$ remains an open question.

The theoretical inclusion relationship between the BD-type and other GDNN cones is of interest when the underlying cone $\mathbb{K}$ involves positive semidefinite cones.
In addition, the case of $\mathbb{K}$ being neither the direct product of a nonnegative orthant and second-order cones nor the direct product of a nonnegative orthant and positive semidefinite cones is also intriguing.
Seemingly, these are more challenging tasks than the cases we treat in this paper.
First, for a general cone $\mathbb{K}$, the set $\mathcal{K}_{{\rm ZVP},0}(\mathbb{K})$ is not necessarily closed, which would preclude us from investigating the strict inclusion relationship between ZVP- and NN-type GDNN cones in the same way as in this paper.
Second, a characterization of the NN-type GDNN cone $\DNN_{\rm NN}(\mathbb{K})$ exists theoretically for a general symmetric cone, as shown in \eqref{def:C0}, which needs a closure operator in general~\cite[Proposition~\mbox{3.6}]{NN20XX}.
Third, when $\mathbb{K}$ involves a matrix cone such as a positive semidefinite cone, the vectorization of a matrix and matrization of a vector are required to deal with the BD-type GDNN cone $\DNN_{{\rm BD}}(\mathbb{K})$.
Consider the case where $\mathbb{K}$ is the positive semidefinite cone $\svec(\mathbb{S}_+^n)$.
In this scenario, a symmetric matrix $\bm{X} \in \mathbb{S}^{T_n}$ belongs to $\DNN_{{\rm BD}}(\svec(\mathbb{S}_+^n))$ only if $\smat(\bm{X}\svec(\bm{v}\bm{v}^\top)) \in \mathbb{S}_+^n$ for all $\bm{v} \in S^{n-1}$.
(Note that any extreme ray of $\mathbb{S}_+^n$ is generated by the matrix $\bm{v}\bm{v}^\top$ for some $\bm{v} \in S^{n-1}$.)
This condition appears to be much more difficult to verify theoretically than the discussion conducted in Section~\ref{subsec:soc}.

\vspace{0.5cm}
\noindent
{\bf Acknowledgments}
This work was supported by the Tokyo Tech Advanced Human Resource Development Fellowship for Doctoral Students and Japan Society for the Promotion of Science Grants-in-Aid for Scientific Research (Grant Numbers JP20H02385 and JP22J20008).

\begin{appendices}
\section{Example of $\mathbb{K}$ Such That $\mathcal{K}_{{\rm ZVP},0}(\mathbb{K})$ Is Not Closed}\label{apdx:example_ZVP_not_closed}
Let $\bm{a} \coloneqq (1,0)^\top$ and $\bm{Q} \coloneqq
\begin{pmatrix}
0&0\\
0&-1
\end{pmatrix}$.
Then,
\begin{equation*}
\mathbb{K} \coloneqq \{\bm{x}\in\mathbb{R}^2 \mid \bm{x}^\top \bm{Q}\bm{x}  = -x_2^2 \ge 0,\ \bm{a}^\top\bm{x} =x_1\ge 0\}
\end{equation*}
is a pointed semialgebraic convex cone (the nonnegative $x_1$-axis) and the vector $\bm{a}$ satisfies $\bm{a} \in \setint(\mathbb{K}^*)$.
It follows from Lemma~\ref{lem:K_ZVP} that $\mathcal{K}_{{\rm ZVP},0}(\mathbb{K}) = \mathbb{S}_+^2 + \mathbb{R}_+\bm{Q}$.
For each $k$, we define $\bm{P}_k \coloneqq \begin{pmatrix}
1/k&1\\
1&k
\end{pmatrix} \in \mathbb{S}_+^2$ and $\bm{X}_k \coloneqq \bm{P}_k + k\bm{Q} \in \mathcal{K}_{\rm ZVP}(\mathbb{K})$.
Then, we have $\bm{X}_k = \begin{pmatrix}
1/k&1\\
1&0
\end{pmatrix} \rightarrow \bm{X}^* \coloneqq \begin{pmatrix}
0&1\\
1&0
\end{pmatrix}$.

To see that $\bm{X}^* \not\in \mathcal{K}_{{\rm ZVP},0}(\mathbb{K})$, we assume that there exist $\bm{P}^* \in \mathbb{S}_+^2$ and $t^* \ge 0 $ such that $\bm{X}^* = \bm{P}^* + t^*\bm{Q}$.
This equation implies that $P_{1,1}^* = 0$ and $P_{1,2}^* = 1$.
However, by $P_{1,1}^* = 0$ and the positive semidefiniteness of $\bm{P}^*$, $P_{1,2}^*$ must be $0$, which is a contradiction.
Therefore, $\mathcal{K}_{{\rm ZVP},0}(\mathbb{K})$ is not closed.

\section{Explicit Exchange Method for BD-Type Generalized Doubly Nonnegative Programming}
\label{apdx:algorithm}
For a symmetric cone $\mathbb{K}$, the algorithm used to solve a BD-type GDNNP problem obtained by replacing $\CP(\mathbb{K})$ in problem~\eqref{prob:GCPP} with $\DNN_{\rm BD}(\mathbb{K})$ is presented in Algorithm~\ref{alg:explicit}.
Although we preliminarily set $\gamma_k$ to $0.5^k$, in the branch of Step 1-2, if case~\eqref{enum:isnot_exist} with $r=0$ occurred five times in a row, $\gamma_{k+1}$ was set to $\tau$ at the end of Step 2 and returned to Step 1 to accelerate convergence. The threshold value $\tau$ in Algorithm~\ref{alg:explicit} was set to $10^{-5}$ and subset $\mathfrak{J}^{(0,0)}$ was set to $\emptyset$.
We tried to identify $\bm{s}_{\rm new}^{(k,r)}$ using the following method.
Before starting the algorithm, we fixed a finite subset $\mathfrak{J}_{\rm fix}$ of a compact set $\mathfrak{J}$ satisfying
\footnote{For example, the set of primitive idempotents in a Euclidean Jordan algebra associated with $\mathbb{K}$ is a compact set satisfying \eqref{eq:Ext_subset} (see \cite[Exercise~\mbox{\RNum{4}.5}]{FK1994} and \cite[Corollary~\mbox{12}]{IL2017}).}
\begin{equation}
\Ext(\mathbb{K}) = \{\alpha\bm{s} \mid \alpha > 0,\ \bm{s}\in \mathfrak{J}\}. \label{eq:Ext_subset}
\end{equation}
In Section~\ref{subsec:misocp}, because
\begin{equation*}
\mathfrak{J} = \{(\bm{e}_i,\bm{0}_n)\mid i = 1,\dots,2n+1\} \cup \{(\bm{0}_{2n+1},1/2,\bm{v}/2)\mid \bm{v}\in S^{n-2}\}
\end{equation*}
satisfies \eqref{eq:Ext_subset}, we took
\begin{equation*}
\mathfrak{J}_{\rm fix} = \{(\bm{e}_i,\bm{0}_n)\mid i = 1,\dots,2n+1\} \cup \{(\bm{0}_{2n+1},1/2,\bm{v}_i/2)\mid i = 1,\dots,1000\},
\end{equation*}
where each $\bm{v}_i$ was generated from $S^{n-2}$ randomly.
In Section~\ref{subsec:max-cut}, because
\begin{multline*}
\mathfrak{J} = \{(1,\bm{0}_{3T_n})\} \cup \{(0,\svec(\bm{v}\bm{v}^\top),\bm{0}_{2T_n}) \mid \bm{v}\in S^{n-1}\}\\ \cup \{(0,\bm{0}_{T_n},\svec(\bm{v}\bm{v}^\top),\bm{0}_{T_n}) \mid \bm{v}\in S^{n-1}\} \cup \{(0,\bm{0}_{2T_n}\svec(\bm{v}\bm{v}^\top)) \mid \bm{v}\in S^{n-1}\}
\end{multline*}
satisfies \eqref{eq:Ext_subset}, we took
\begin{multline*}
\mathfrak{J}_{\rm fix} = \{(1,\bm{0}_{3T_n})\} \cup \{(0,\svec(\bm{v}_i\bm{v}_i^\top),\bm{0}_{2T_n}) \mid i=1,\dots,1000\}\\ \cup \{(0,\bm{0}_{T_n},\svec(\bm{v}_i\bm{v}_i^\top),\bm{0}_{T_n}) \mid i=1001,\dots,2000\}\\
\cup \{(0,\bm{0}_{2T_n}\svec(\bm{v}_i\bm{v}_i^\top)) \mid i=2001,\dots,3000\},
\end{multline*}
where each $\bm{v}_i$ was generated from $S^{n-1}$ randomly.
In Step 1-2, we first checked whether there exists $\bm{s}_{\rm new} \in \mathfrak{J}_{\rm fix}$ such that $\lambda_{\rm min}(\bm{X}^{(k,r)}\bm{s}_{\rm new} + \gamma_k\bm{e}) < 0$ where $\lambda_{\rm min}(\cdot)$ is the minimum eigenvalue in the sense of Euclidean Jordan algebras~\cite[Theorem~\mbox{\RNum{3}.1.1}]{FK1994}.
If such $\bm{s}_{\rm new}$ exists, we let $\bm{s}_{\rm new}^{(k,r)} = \bm{s}_{\rm new}$.
Otherwise, we next checked whether the optimal value of the optimization problem $\min_{\bm{s}\in \mathfrak{J}}\lambda_{\rm min}(\bm{X}^{(k,r)}\bm{s} + \gamma_k\bm{e})$ is less than 0.
In addition, $\bm{v}_{\bm{s}}^{(k,r+1)}$ was regarded as $\bm{0}$ if $\|\bm{v}_{\bm{s}}^{(k,r+1)}\| \le 10^{-12}$.

\begin{algorithm}[t]
\caption{Explicit exchange method for BD-type GDNNP}
\label{alg:explicit}
\begin{description}
\item[\bf Step 0.]Choose a positive sequence $\{\gamma_k\}_k$ such that $\gamma_k \downarrow 0\ (k\to\infty)$, a positive value $\tau$, a compact subset $\mathfrak{J}$ of $\Ext(\mathbb{K})$ such that $\Ext(\mathbb{K}) = \{\alpha\bm{s} \mid \alpha > 0,\ \bm{s}\in \mathfrak{J}\}$, and a finite subset $\mathfrak{J}^{(0,0)}$ of $\mathfrak{J}$.
Solve ${\rm P}(\mathfrak{J}^{(0,0)})$ to obtain an optimal solution $\bm{X}^{(0,0)}$, where ${\rm P}(\mathfrak{J}')$ is defined as
\begin{equation*}
{\rm P}(\mathfrak{J}'):\quad
\begin{alignedat}{3}
&\minimize_{\bm{X}} && \quad \langle \bm{C},\bm{X}\rangle \\
&\subjectto && \quad \langle \bm{A}_i, \bm{X}\rangle = b_i\ (i=1,\dots,m),\\
&&&\quad \bm{X}\in \mathbb{S}_+^n,\\
&&&\quad \bm{X}\bm{s} \in \mathbb{K}\ (\forall\bm{s}\in \mathfrak{J}')
\end{alignedat}
\end{equation*}
for a finite subset $\mathfrak{J}'$ of $\mathfrak{J}$.
Set $k \coloneqq 0$.
\item[\bf Step 1.]Obtain $\bm{X}^{(k+1,0)}$ and $\mathfrak{J}^{(k+1,0)}$ by the following procedure.
\begin{description}
\item [\bf Step 1-1.]Set $r\coloneqq 0$.
\item [\bf Step 1-2.]Try to find $\bm{s}_{\rm new}^{(k,r)} \in \mathfrak{J}$ such that $\lambda_{\rm min}(\bm{X}^{(k,r)}\bm{s}_{\rm new}^{(k,r)} + \gamma_k\bm{e}) < 0$.
\begin{enumerate}[(a)]
\item If there exists such $\bm{s}_{\rm new}^{(k,r)}$, let $\overbar{\mathfrak{J}}^{(k,r+1)} \coloneqq \mathfrak{J}^{(k,r)}\cup\{\bm{s}_{\rm new}^{(k,r)}\}$ and solve ${\rm P}(\overbar{\mathfrak{J}}^{(k,r+1)})$ to find an optimal solution $\bm{X}^{(k,r+1)}$ and optimal dual variables $\bm{v}_{\bm{s}}^{(k,r+1)}\ (\bm{s}\in\overbar{\mathfrak{J}}^{(k,r+1)})$. Let $\mathfrak{J}^{(k,r+1)}\coloneqq \{\bm{s}\in\overbar{\mathfrak{J}}^{(k,r+1)}\mid \bm{v}_{\bm{s}}^{(k,r+1)}\neq\bm{0}\}$. Set $r \coloneqq r + 1$ and go to the beginning of Step 1-2; \label{enum:is_exist}
\item Otherwise, let $\bm{X}^{(k+1,0)} \coloneqq \bm{X}^{(k,r)}$ and $\mathfrak{J}^{(k+1,0)} \coloneqq \mathfrak{J}^{(k,r)}$.\label{enum:isnot_exist}
\end{enumerate}
\end{description}
\item[\bf Step 2.]If $\gamma_k \le \tau$, stop the algorithm. Otherwise, set $k \coloneqq k + 1$ and go to Step 1.
\end{description}
\end{algorithm}
\end{appendices}


\begin{thebibliography}{99}
\bibitem{BMP2016}L. Bai, J.E. Mitchell, and J.-S. Pang: On conic QPCCs, conic QCQPs and completely positive programs. \textit{Mathematical Programming}, {\bf 159} (2016), 109--136.
\bibitem{BN2001}A. Ben-Tal and A. Nemirovski: \textit{Lectures on Modern Convex Optimization: Analysis, Algorithms, and Engineering Applications} (SIAM, Philadelphia, PA, 2001).
\bibitem{BDde2000}I.M. Bomze, M. D\"{u}r, E. de Klerk, C. Roos, A.J. Quist, and T. Terlaky: On copositive programming and standard quadratic optimization problems. \textit{Journal of Global Optimization}, \textbf{18} (2000), 301--320.
\bibitem{BV2004}S. Boyd and L. Vandenberghe: \textit{Convex Optimization} (Cambridge University Press, Cambridge, 2004).
\bibitem{Burer2009}S. Burer: On the copositive representation of binary and continuous nonconvex quadratic programs. \textit{Mathematical Programming}, \textbf{120} (2009), 479--495.
\bibitem{Burer2012}S. Burer: Copositive programming. In M.F. Anjos and J.B. Lasserre (eds.): \textit{Handbook on Semidefinite, Conic and Polynomial Optimization} (Springer, Boston, MA, 2012), 201--218.
\bibitem{BD2012}S. Burer and H. Dong: Representing quadratically constrained quadratic programs as generalized copositive programs. \textit{Operations Research Letters}, \textbf{40-3} (2012), 203--206.
\bibitem{FK1994}J. Faraut and A. Kor\'{a}nyi: \textit{Analysis on Symmetric Cones} (Clarendon Press, Oxford, 1994).
\bibitem{GW1995}M.X. Goemans and D.P. Williamson: Improved approximation algorithms for maximum cut and satisfiability problems using semidefinite programming. \textit{Journal of the Association for Computing Machinery}, \textbf{42-6} (1995), 1115--1145.
\bibitem{GS2013}M.S. Gowda and R. Sznajder: On the irreducibility, self-duality, and non-homogeneity of completely positive cones. \textit{Electronic Journal of Linear Algebra}, \textbf{26} (2013), 177--191.
\bibitem{Gurobi}Gurobi: Gurobi Optimizer Quick Start Guide. \url{https://www.gurobi.com/} (accessed 2023-12-14).
\bibitem{IL2017}M. Ito and B.F. Louren\c{c}o: A bound on the Carath\'{e}odory number. \textit{Linear Algebra and its Applications}, \textbf{532} (2017), 347--363.
\bibitem{Karp1972}R.M. Karp: Reducibility among combinatorial problems. In R.E. Miller, J.W. Thatcher, and J.D. Bohlinger (eds.): \textit{Complexity of Computer Computations} (Springer, Boston, MA, 1972), 85--103.
\bibitem{Lofberg2004}J. L\"{o}fberg: YALMIP: a toolbox for modeling and optimization in MATLAB. In: \textit{Proceedings of the 2004 IEEE International Symposium on Computer Aided Control Systems Design} (IEEE, 2004), 284--289.
\bibitem{MT2015}R. Miyashiro and Y. Takano: Mixed integer second-order cone programming formulations for variable selection in linear regression. \textit{European Journal of Operational Research}, \textbf{247-3} (2015), 721--731.
\bibitem{MOSEK}Mosek: MOSEK optimization toolbox for MATLAB. \url{https://www.mosek.com/} (accessed 2023-12-14).
\bibitem{NN2022}M. Nishijima and K. Nakata: A block coordinate descent method for sensor network localization. \textit{Optimization Letters}, \textbf{16} (2022), 1051--1071.
\bibitem{NN20XX}M. Nishijima and K. Nakata: Approximation hierarchies for copositive cone over symmetric cone and their comparison. \textit{Journal of Global Optimization}, to appear. \url{doi:10.1007/s10898-023-01319-3}.
\bibitem{OHF2012}T. Okuno, S. Hayashi, and M. Fukushima: A regularized explicit exchange method for semi-infinite programs with an infinite number of conic constraints. \textit{SIAM Journal on Optimization}, \textbf{22-3} (2012), 1009--1028.
\bibitem{Parrilo2000}P.A. Parrilo: \textit{Structured Semidefinite Programs and Semialgebraic Geometry Methods in Robustness and Optimization} (Ph.D. thesis, California Institute of Technology, Pasadena, CA, 2000).
\bibitem{PT2007}I. P\'{o}lic and T. Terlaky: A survey of the S-lemma. \textit{SIAM Review}, \textbf{49-3} (2007), 371--418.
\bibitem{PR2009}J. Povh and F. Rendl: Copositive and semidefinite relaxations of the quadratic assignment problem. \textit{Discrete Optimization}, \textbf{6-3} (2009), 231--241.
\bibitem{PH2018}M.N. Prasad and G.A. Hanasusanto: Improved conic reformulations for $K$-means clustering. \textit{SIAM Journal on Optimization}, \textbf{28-4} (2018), 3105--3126.
\bibitem{Rockafellar1970}R.T. Rockafellar: \textit{Convex Analysis} (Princeton University Press, Princeton, NJ, 1970).
\bibitem{SB2021}N. Shaked-Monderer and A. Berman: \textit{Copositive and Completely Positive Matrices} (World Scientific, Singapore, 2021).
\bibitem{SZ2003}J.F. Sturm and S. Zhang: On cones of nonnegative quadratic functions. \textit{Mathematics of Operations Research}, \textbf{28-2} (2003), 246--267.
\bibitem{YM2010}A. Yoshise and Y. Matsukawa: On optimization over the doubly nonnegative cone. In: \textit{Proceedings of 2010 IEEE Multi-conference on Systems and Control} (IEEE, 2010), 13--18.
\bibitem{ZVP2006}L.F. Zuluaga, J. Vera, and J. Pe\~{n}a: LMI approximations for cones of positive semidefinite forms. \textit{SIAM Journal on Optimization}, \textbf{16-4} (2006), 1076--1091.
\end{thebibliography}
\end{document}